\documentclass[11pt, oneside]{paper} 
\usepackage{amsmath, amsthm, amssymb,geometry,tikz-cd,color,enumerate}
\usepackage{enumitem}
\usepackage{quiver}
\usepackage{mathrsfs}
\usepackage[title]{appendix}

\usepackage{tocloft}
\definecolor{mylinkcolor}{rgb}{0.05,0.05,0.4}
\definecolor{mycitecolor}{rgb}{0.1,0.1,0.8}
\definecolor{myurlcolor}{rgb}{0.1,0.1,0.8}
\usepackage{hyperref}
\hypersetup{colorlinks,linkcolor=mylinkcolor,citecolor=mycitecolor,urlcolor=myurlcolor, linktocpage}

\newcommand{\RV}{\mathrm{RV}}
\newcommand{\M}{\mathrm{M}}
\newcommand{\bfomega}{\mathbf{\Omega}}

\newtheorem{thm}{Theorem}[section]

\newtheorem{lem}[thm]{Lemma}
\newtheorem{cor}[thm]{Corollary}
\newtheorem{prop}[thm]{Proposition}

\theoremstyle{definition}

\newtheorem{rem}[thm]{Remark}
\newtheorem{eg}[thm]{Example}

\title{A categorical treatment of the Radon-Nikodym theorem and martingales}
\author{Ruben Van Belle \thanks{School of Mathematics, University of Edinburgh; ruben.van.belle@ed.ac.uk}}
\begin{document}
\maketitle
\begin{abstract}

In this paper we will give a categorical proof of the Radon-Nikodym theorem. We will do this by
describing the trivial version of the result on finite probability spaces as a natural isomorphism. We then proceed to Kan extend this isomorphism to obtain the result for general probability spaces. Moreover, we observe that conditional expectation naturally appears in the construction of the right Kan extensions.
Using this we can represent martingales, a special type of stochastic processes, categorically.

We then repeat the same construction for the case where everything is enriched over $\mathbf{CMet}$, the category of complete metric spaces and $1$-Lipschitz maps. In the enriched context, we can give a categorical proof
of a martingale convergence theorem, by showing that a certain functor preserves certain cofiltered limits.
\end{abstract}

\smalltableofcontents
\section{Introduction}\label{RNSection1}

The Radon-Nikodym theorem gives a correspondence between random variables and measures, two important concepts in probability theory and measure theory. If we fix a probability space $(\Omega,\mathcal{F},\mathbb{P})$, then we can look at the \emph{random variables} on this probability space and the measures that are \emph{absolutely continuous} with respect to $\mathbb{P}$. The Radon-Nikodym theorem tells us that there is a canonical correspondence between these two. The random variable associated to a measure $\mu$ that is absolutely continuous with respect to $\mathbb{P}$ is called the \emph{Radon-Nikodym derivative with respect to $\mathbb{P}$}. The classical proof gives a concrete construction of these Radon-Nikodym derivatives and relies on the Hahn decomposition theorem, e.g. Theorem 31.B in \cite{halmos}, Theorem 4.2.2 in \cite{cohn} and Theorem 3.2.2 in \cite{bogachev}. Important applications of this result are the existence of conditional expectation and the Girsanov theorem in stochastic calculus.

In section \ref{RNSection2} of this paper, we will give a \emph{categorical proof} for this result. Moreover, we will not only prove that there is a bijection between random variables and measures, but also that this bijection is an \emph{isometry}. We will do this by starting with the trivial case, when $\Omega$ is finite. We translate this categorically as a natural isomorphism between certain functors. We will then proceed by \emph{Kan extending} the trivial, finite version of the result to the general result. This happens in two parts. The first part (section \ref{RNSection2.2}) is straightforward and purely categorical, not relying on any results in measure theory. The second part of the proof (section \ref{RNSection2.3}) does require some measure theory, in particular it relies on the Riesz-Fischer theorem (Theorem \ref{RN2.1}). Furthermore, the concept of conditional expectation naturally arises from the Kan extension construction in section \ref{RNSection2.3}.

In section \ref{RNSection3}, we will focus on \emph{martingales}, a special class of stochastic processes. Important examples of martingales are Brownian motion and unbiased random walks. Furthermore, martingales have nice convergence properties, which are described by \emph{Doob's martingale convergence theorem}. The proof of this result relies on stopping times, the optional stopping theorem and several lemmas about upcrossings by stochastic processes. The original proof can be found in section XI.14 in \cite{doob}. We will give a categorical proof of a weaker version of this result. We do this by showing that a certain class of functors preserve certain cofiltered limits (Theorem \ref{RN3.16}). By applying this to the functors representing random variables from section \ref{RN3.3}, we immediately obtain a proof for a weaker martingale convergence theorem. Moreover, if we apply the same result to the functor representing measures from section \ref{RN3.3}, we find a Kolmogorov extension-type theorem. In this section we use the results from section \ref{RNSection2}. However, we first lift everything to the \emph{enriched} setting. We consider everything to be enriched over the category of \emph{complete metric spaces}. This part is crucial to obtain the main result in this section (Theorem \ref{RN3.16}). 

In this paper we will consider every distance function to be an \emph{extended pseudometric}, i.e. a function $d:X\times X\to [0,\infty]$ on a set $X$ such that for all $x,y\in X$, $d(x,x)=0$ and $d(x,y)=d(y,x)$, satisfying the triangle-inequality. However,we will refer to an extended pseudometric $d$ as just ‘\emph{metric}' and to the pair $(X,d)$ as just ‘\emph{metric space}'. The category of these metric spaces with the $1$-Lipschitz maps between them is denoted by $\textbf{Met}$ and the full subcategory of \emph{complete} metric spaces by $\textbf{CMet}$.

For a measurable map $f: X\to Y$ and a measure $\mu$ on $X$, we write $\mu\circ f^{-1}$ to mean the \emph{pushforward measure} of $\mu$ along $f$, i.e. for every measurable subset $E$ of $Y$, $$\left(\mu\circ f^{-1}\right)(E):=\mu(f^{-1}(E)).$$

Furthermore, for two real number $a$ and $b$, the minimum of the two is denoted by $a\wedge b$, while the maximum of the two is denoted by $a\vee b$.
    
\section{The Radon-Nikodym theorem}\label{RNSection2}
The Radon-Nikodym theorem is an important result in measure theory and probability theory. The goal of this section is to give a \emph{categorical} proof of this result. We will start this section by recalling the result and explaining its applications in probability theory. We will then discuss the trivial \emph{finite} version of the Radon-Nikodym theorem and explain how this translate categorically. After that, we focus on ‘\emph{Kan extending}' the trivial finite result to the general result. 

Probability theory studies \emph{random variables} and \emph{probability distributions}. The connection between these two is given by the Radon-Nikodym theorem. 

Let $\bfomega:=(\Omega,\mathcal{F},\mathbb{P})$ be a probability space. We say that two measurable functions $f,g:\Omega\to [0,\infty)$ are $\mathbb{P}$-almost surely equal if $\mathbb{P}(f=g)=1$ and we write $f=_\mathbb{P}g$. This defines an equivalence relation on the set $\textbf{Mble}(\Omega,[0,\infty))$ of measurable functions $X\to [0,\infty)$. 
The set $$\textbf{Mble}(\Omega,[0,\infty))/=_{\mathbb{P}}$$ 
becomes a metric space, by endowing it with the $L^1$\textit{-metric} defined by $$d_{L^1}(f,g):=\int \lvert f-g\rvert \text{d}\mathbb{P},$$ for all $f,g \in \textbf{Mble}(\Omega,[0,\infty))/=_{\mathbb{P}}$. We will denote this space of random variables by $\mathrm{RV}(\bfomega)$. For a real number $r>0$, let $\mathrm{RV}_r(\bfomega)$ be the subspace of random variable $f$ such that $\mathbb{P}(f\leq r)=1$.

An important result about the space of random variables that we will need later is the Riesz-Fischer theorem. 

\begin{thm}[Riesz-Fischer]\label{RN2.1}
Let $\bfomega:=(\Omega,\mathcal{F},\mathbb{P})$ be a probability space and let $r>0$ be a natural number. The metric spaces $\RV(\bfomega)$ and $\RV_r(\bfomega)$ are complete.
\end{thm}
\begin{proof}
This is Theorem 2.2 in \cite{stein}.
\end{proof}

For a measure on $\mu$ on a measure space $(\Omega,\mathcal{F})$, we say that $\mu$ \textit{is absolutely continuous with respect to }$\mathbb{P}$ if $\mu(A)=0$ for all measurable subsets $A\subseteq \Omega$ such that $\mathbb{P}(A)=0$; this is denoted as $\mu\ll \mathbb{P}$. The set of measures on $(\Omega,\mathcal{F})$ that are absolutely continuous with respect to $\mathbb{P}$ becomes a metric space, by endowing it with the \textit{total variation metric} defined by \begin{align*} d_{\mathrm{TV}}(\mu,\nu) &:=\lvert \mu-\nu\rvert(\Omega) \\ &= \sup\left\{\sum_{n=1}^\infty \lvert \mu(A_n)-\nu(A_n)\rvert \mid (A_n)_{n=1}^\infty \text{ measurable partition of }\Omega\right\},\end{align*} for all $\mu, \nu \ll \mathbb{P}$. We denote this space of measures by $\mathrm{M}(\bfomega)$. For a real number $r>0$, we write $\mathrm{M}_r(\bfomega)$ for the subspace of measures $\mu$ on $(\Omega,\mathcal{F})$ such that $\mu\leq r\mathbb{P}$, i.e. $\mu(A)\leq r\mathbb{P}(A)$ for all $A\in \mathcal{F}$. 

We have the following well-known result, for which we will give a short proof. A proof can also be found in section III.7.4 in \cite{dunford}.

\begin{prop}\label{RN2.2}
    Let $\bfomega:=(\Omega,\mathcal{F},\mathbb{P})$ be a probability space and let $r>0$ be a natural number. The metric spaces $\M(\bfomega)$ and $\M_r(\bfomega)$ are complete.
\end{prop}
\begin{proof}
    Let $(\mu_n)_n$ be a Cauchy sequence in $\M(\bfomega)$. Since $\lvert \mu_p(A)-\mu_q(A)\rvert \leq d_{\mathrm{TV}}(\mu_p,\mu_q)$ for all $A\in \mathcal{F}$, it follows that $(\mu_n(A))_n$ is a Cauchy sequence in $[0,\infty)$. Define a map $\mu:\mathcal{F}\to [0,\infty)$ by $$A\mapsto \lim_{n\to \infty}\mu_n(A).$$

    It is clear that $\mu$ is finitely additive. 
    Let $(A_n)_{n=1}^\infty$ be a decreasing sequence of measurable subsets such that $A_n\downarrow \emptyset$.

    For $\epsilon>0$, there is an $N$ such that for $d_{\mathrm{TV}}(\mu_{p},\mu_{q})<\frac{\epsilon}{2}$ for all $p, q\geq N$. Since $\mu_N$ is $\sigma$-additive, there is an $M$ such that for $m\geq M$, $\mu_N(A_m)<\frac{\epsilon}{2}$. Therefore, for $n\geq N$ and $m\geq M$, we have $$\mu_n(A_m)\leq \mu_N(A_m)+d_{\mathrm{TV}}(\mu_n,\mu_N) \leq \epsilon. $$

    Taking $n\to\infty$, shows that $\mu(A_m)<\epsilon$ and therefore $\lim_{m\to \infty}\mu(A_m)=0$. We conclude that $\mu$ is $\sigma$-additive. 

    We finish the proof by showing that $(\mu_n)_n$ converges to $\mu$ in $\M(\bfomega)$. Let $K\geq 1$ be a natural number. 

    Then for every countable partition $(A_k)_{k=1}^\infty$ and $p,q\geq N$, $\sum_{k=1}^K\lvert \mu_{p}(A_k)-\mu_{q}(A_k)\rvert \leq d_{\mathrm{TV}}(\mu_p,\mu_q)<\frac{\epsilon}{2}.$

    Taking first $p\to \infty$ and then $K\to \infty$ gives that $$\sum_{k=1}^\infty \lvert \mu(A_k)-\mu_q(A_k)\rvert <\frac{\epsilon}{2},$$ for every countable partition $(A_k)_{k=1}^\infty$ and $q\ge N$. 

    By taking the supremum over all countable partitions, we find that $$d_{\mathrm{TV}}(\mu,\mu_q)<\frac{\epsilon}{2}$$ for all $q\geq N$. This shows that $\mu_n\to \mu$ in $\M(\bfomega)$. 
\end{proof}

The space of random variables and the space of measures on $\bfomega$ are connected in the following way. For a random variable $f\in \RV(\bfomega)$, define a measure $\varphi(f)$ on $(\Omega,\mathcal{F})$ by the assignment $$A\mapsto \int_Af\text{d}\mathbb{P}.$$ If $\mathbb{P}(A)=0$, then $\varphi(f)(A)=0$ and therefore $\varphi(f) \in \M(\bfomega)$. Moreover for $f,g\in \RV(\bfomega)$ and a measurable partition $(A_n)_{n=1}^\infty$ of $\Omega$, 

$$\sum_{n=1}^\infty \left \lvert \int_{A_n}f\text{d}\mathbb{P}-\int_{A_n}g\text{d}\mathbb{P}\right\rvert \leq \sum_{n=1}^\infty \int_{A_n}\lvert f-g\rvert \text{d}\mathbb{P} = \int\lvert f-g\rvert \text{d}\mathbb{P}.$$ Taking the supremum over all such partitions gives that $d_{\mathrm{TV}}(\varphi(f),\varphi(g))\leq d_{L^1}(f,g)$. Therefore the map $\varphi:\RV(\bfomega)\to \M(\bfomega)$ is $1$-Lipschitz.

\begin{thm}[Radon-Nikodym]\label{RN2.3}
The map $\varphi$ is an isomorphism. In particular, for every measure $\mu$ such that $\mu \ll \mathbb{P}$ there exists a $\mathbb{P}$-almost surely unique measurable map $f:\Omega\to [0,\infty)$ such that  for all $A\in \mathcal{F}$,  
$$\mu(A)=\int_Af\text{d}\mathbb{P}.$$
\end{thm}

The $f$ in Theorem \ref{RN2.3} is called the \textbf{Radon-Nikodym derivative of $\mu$ with respect to $\mathbb{P}$} and is denoted as $\frac{\text{d}\mu}{\text{d}\mathbb{P}}$.

The following example is an application of the Radon-Nikodym theorem in probability theory. The existence of \emph{conditional expectation} can be proven using this result. The concept of conditional expectation is important in \emph{martingale} theory, which we will discuss further in Section \ref{RNSection3}. 
\begin{eg}[Conditional expectation]\label{RN2.4}
Consider two probability spaces $\bfomega_1:=(\Omega_1,\mathcal{F}_1,\mathbb{P}_1)$ and $\bfomega_2:=(\Omega_2,\mathcal{F}_2,\mathbb{P}_2)$ and let $g:\bfomega_1\to \bfomega_2$ be a measure preserving map, i.e. $\mathbb{P}_1\circ g^{-1}=\mathbb{P}_2$.

For a random variable $f\in \RV(\bfomega_1)$, we can define a measure $\mu$ on $\bfomega_2$ as $$\mu(B)=\int_{g^{-1}(B)}f\text{d}\mathbb{P}_1$$for all $B\in \mathcal{F}_2$. If $\mathbb{P}_2(B)=0$, then $\mathbb{P}_1(g^{-1}(B))=0$ and therefore $\mu\in \M(\bfomega_2)$. 

By the Radon-Nikodym theorem (Theorem \ref{RN2.3}), there exists a unique $\tilde{f}\in \RV(\bfomega_2)$ such that \begin{equation}\label{eq1}
    \int_{g^{-1}(B)}f\text{d}\mathbb{P}_1=\mu(B)=\int_B\tilde{f}\text{d}\mathbb{P}_2.\end{equation}
The random variable $\Tilde{f}$ is called the \textbf{conditional expectation of $f$ with respect to $g$} and is denoted as $\mathbb{E}[f\mid g]$. Because of the uniqueness in the Radon-Nikodym theorem, equation (\ref{eq1}) is the defining property for the random variable $\mathbb{E}[f\mid g]$.
\end{eg}
\subsection{The finite Radon-Nikodym theorem}\label{RNSection2.1}
In the case that we are working with a \emph{finite} probability space, the Radon-Nikodym becomes trivial. We will discuss this trivial version in this section and explain how this can be expressed categorically. To do this we will define two functors, one expressing random variables and one expressing measures. The finite version of the Radon-Nikodym theorem then corresponds to saying that these functors are isomorphic. 

A \textbf{finite probability space} is a probability space whose underlying set $A$ is finite and whose $\sigma$-algebra is the whole powerset $\mathcal{P}(A)$. We will write $(A,p)$ instead of $(A,\mathcal{P}(A),p)$. For an element $a\in A$, we write $p_a$ or $p(a)$ to mean $p(\{a\})$.

We denote the category of probability spaces and measure-preserving maps by $\textbf{Prob}$ and the full subcategory of finite probability measure by $\textbf{Prob}_f$. The inclusion functor $\textbf{Prob}_f\to \textbf{Prob}$ is denoted by $i$.

We start by defining a functor of measures $\M^f:\textbf{Prob}_f\to \textbf{Met}$. This functor sends a finite probability space $(A,p)$ to $\M(A,p)$ and a measure preserving map $s:(A,p)\to (B,q)$ of finite probability spaces to the $1$-Lipschitz map $\M^f(s):\M(A,p)\to \M(B,q)$, which is defined by the assignment
$$m\mapsto m\circ s^{-1}.$$ Similarly we can define a functor $\M^f_r:\textbf{Prob}_f\to \textbf{Met}$ for every positive real number $r$.

We can also define a functor of random variables $\RV^f:\textbf{Prob}_f\to \textbf{Met}$ in the following way. On objects this functor is defined by sending a finite probability space $(A,p)$ to the metric space $\RV(A,p)$. On morphisms this functor sends a measure preserving map $s:(A,p)\to (B,q)$ of finite probability spaces to the $1$-Lipschitz map $\RV^f(s):\RV(A,p)\to \RV(B,q)$ which is defined for $g\in \RV(A,p)$ by $$\RV^f(s)(g):(B,q)\to [0,\infty): b\mapsto \begin{cases}
    \frac{1}{q_b}\sum_{s(a)=b}p_qg(a) \text{ if }q_b\not=0\\ 0\text{ otherwise.}
\end{cases}$$
The map $\RV^f(s)(g)$ does not depend on the representation of $g$ and therefore it is well-defined. In a similar way we define functors $\RV^f_r:\textbf{Prob}_f\to \textbf{Met}$ for real numbers $r>0$.

For a finite probability space $(A,p)$, we define the $1$-Lipschitz map $(\rho^f_r)_A:\RV_r^f(A,p)\to \M^f_r(A,p)$ by the assignment $$g\mapsto (g(a)p_a)_{a\in A}.$$

The finite version of the Radon-Nikodym theorem can now be expressed in the following way.
\begin{prop}[Finite bounded Radon-Nikodym]\label{RN2.5}
The maps $((\rho^f_r)_A)_{(A,p)}$ form a natural isomorphism $\rho^f_r:\RV_r^f\to \M_r^f$.
\end{prop}
\begin{proof}
It is easy to see that $(\rho^f_r)_A$ is well-defined and invertible for every finite probability space $(A,p)$. If follows now by Proposition $4.2$ in \cite{levin} that this is an isomorphism of metric spaces. It is straightforward to check the naturality of $\rho^f_r$.
\end{proof}

Since $\RV_r^f$ and $\M_r^f$ are isomorphic by Proposition \ref{RN2.5}, so are their right Kan extensions along the inclusion $i:\textbf{Prob}_f\to \textbf{Prob}$. 

\[\begin{tikzcd}
	{\textbf{Prob}_f} && {\textbf{Met}} \\
	\\
	{\textbf{Prob}}
	\arrow[""{name=0, anchor=center, inner sep=0}, "{\mathrm{M}_r^f}"{pos=0.4}, curve={height=-6pt}, from=1-1, to=1-3]
	\arrow[""{name=1, anchor=center, inner sep=0}, "{\mathrm{RV}_r^f}"'{pos=0.4}, shift right=1, curve={height=6pt}, from=1-1, to=1-3]
	\arrow["i"', from=1-1, to=3-1]
	\arrow[""{name=2, anchor=center, inner sep=0}, curve={height=-6pt}, dashed, from=3-1, to=1-3]
	\arrow[""{name=3, anchor=center, inner sep=0}, shift right=1, curve={height=6pt}, dashed, from=3-1, to=1-3]
	\arrow["\cong"{description}, draw=none, from=0, to=1]
	\arrow["\cong"{description}, draw=none, from=2, to=3]
\end{tikzcd}\]

In the following two sections, we will study the right Kan extensions of these functors. In section \ref{RNSection2.2} we will describe what the right Kan extension of the finite measures functor $\M_r^f:\textbf{Prob}_f\to \textbf{Met}$ along $i:\textbf{Prob}_f\to \textbf{Prob}$ looks like and in section \ref{RNSection2.3} we will do the same for the finite random variables functor $\RV_r^f$. 

This will lead to a categorical proof for the bounded Radon-Nikodym theorem in section \ref{RNSection2.4}.

\subsection{\texorpdfstring{The measures functor $\M$}{The measures functor M}}\label{RNSection2.2}


In this section we will study the right Kan extension of the functor $\M^f_r:\textbf{Prob}_f\to \textbf{Met}$ along the inclusion functor $i:\textbf{Prob}_f\to \textbf{Prob}$. 

\[\begin{tikzcd}
	{\textbf{Prob}_f} && {\textbf{Met}} \\
	\\
	{\textbf{Prob}}
	\arrow["i"', from=1-1, to=3-1]
	\arrow["{\mathrm{M}_r^f}", from=1-1, to=1-3]
	\arrow["{\mathrm{Ran}_i\mathrm{M}_r^f}"', dashed, from=3-1, to=1-3]
\end{tikzcd}\]

We will first describe how $\mathrm{Ran}_i\M_r^f$ acts on objects in Theorem \ref{RN2.6} and then how it acts on morphisms in Proposition \ref{RN2.7}. It will turn out that this right Kan extension expresses certain measures on arbitrary probability spaces. This will then motivate the notation $\M_r:=\mathrm{Ran}_i\M_r^f$.

We will then show that these functors form a diagram $D_\M$:
\[\begin{tikzcd}
	\ldots & {\mathrm{M}_1} & \ldots & {\mathrm{M}_2} & \ldots & {\mathrm{M}_r} & \ldots
	\arrow[from=1-1, to=1-2]
	\arrow[from=1-2, to=1-3]
	\arrow[from=1-3, to=1-4]
	\arrow[from=1-4, to=1-5]
 \arrow[from=1-5, to=1-6]
  \arrow[from=1-6, to=1-7]
\end{tikzcd}\]
In the second part of this section we will describe the colimit of this diagram. The obtained colimiting functor will express measures and therefore this will motivate the notation $\M:\textbf{Prob}\to \mathbf{CMet}$ for the colimit of $D_\M$.

\begin{thm}\label{RN2.6}
Let $\mathbf{\Omega}:=(\Omega,\mathcal{F},\mathbb{P})$ be a probability space. Then
\[\textup{Ran}_i\M^f_r(\bfomega)=\M_r(\bfomega).\]
\end{thm}
\begin{proof}
Let $U:\bfomega \downarrow i \to \textbf{Prob}_f$ be the forgetful functor and let $D_\bfomega$ denote the diagram \[\bfomega \downarrow i \xrightarrow{U} \textbf{Prob}_f\xrightarrow{\M^f_r}\mathbf{CMet}.\]
We will now show that $\M_r(\bfomega)=\lim D_\bfomega$. 

For a measure preserving map $f:\bfomega\to \mathbf{A}$, where $\mathbf{A}:=(A,p)$ is some finite probability space, define a map $p_f:\M_r(\bfomega)\to \M_r^f(\mathbf{A})$ by \[p_f(\mu):=\mu\circ f^{-1}.\]
It can be checked that this map is well-defined and $1$-Lipschitz. It is also straightforward to check that the metric space $\M_r(\bfomega)$ together with the maps $(p_f)_f$ form a cone over the diagram $D_\bfomega$.

We will now show that this cone is universal. To do that, consider another cone $(Y,(q_f)_f))$ over the diagram $D_\bfomega$.

Let $E$ be a measurable subset of $\bfomega$. Let $\textbf{2}:=\{0,1\}$ and let $p_E$ be the probability measure on $\textbf{2}$ defined by \[p_E(1):=\mathbb{P}(E).\]
The assignment \[\omega\mapsto \begin{cases}1 \text{ if }\omega\in E\\
0 \text{ otherwise,}\end{cases}\]
defines a measure preserving map $1_E:\bfomega\to (\textbf{2},p_E).$

For $y\in Y$, define a map $\mu_y:\Sigma \to [0,\infty)$ by \[\mu_y(A):=q_{1_A}(y)_1.\] since $q_{1_E}(y)$ is an element of $\M_r^f(\textbf{2},p_E)$, \[\mu_y(E)=q_{1_E}(y)(1)\leq rp_E(1)=r\mathbb{P}(E)\] and therefore $\mu_y\leq r\mathbb{P}$.

We will now show that $\mu_y$ is a measure. Consider disjoint measurable subsets $E_1$ and $E_2$ of $\Omega$. Let $\textbf{3}:=\{0,1,2\}$ and let $p_{E_1,E_2}$ be the probability measure on $\textbf{3}$, defined by \[p_{E_1,E_2}(1):=\mathbb{P}(E_1) \quad \text{ and }\quad p_{E_1,E_2}(2):=\mathbb{P}(E_2).\]

The assignment \[\omega\mapsto \begin{cases}1 \text{ if }\omega\in E_1\\
2 \text{ if }\omega\in E_2\\
0 \text{ otherwise,}\end{cases}\]
defines a measure preserving map $1_{E_1,E_2}:\bfomega\to (\textbf{3},p_{E_1,E_2}).$ 
Let $s:\textbf{3}\to \textbf{2}$ be the map that fixes $0$. The map $\textbf{3}\to \textbf{2}$ that sends $1$ to $1$ and the other elements to $0$ is denoted by $s_1$. The map $s_2:\textbf{3}\to \textbf{2}$ is defined in a similar way.

We have the following commutative triangles: 
\[\begin{tikzcd}
	& {(X,\Sigma,\mathbb{P})} &&& {(X,\Sigma,\mathbb{P})} \\
	&&& {} \\
	{(\textbf{2},p_{E_2})} & {(\textbf{3},p_{E_1,E_2})} & {(\textbf{2},p_{E_1})} & {(\textbf2,p_{E_1\cup E_2})} && {(\textbf{3},p_{E_1,E_2})}
	\arrow["{1_{E_1,E_2}}"{description}, from=1-2, to=3-2]
	\arrow["{1_{E_1}}"{description}, from=1-2, to=3-3]
	\arrow["{s_1}"', from=3-2, to=3-3]
	\arrow["{1_{E_2}}"{description}, from=1-2, to=3-1]
	\arrow["{s_2}", from=3-2, to=3-1]
	\arrow["{1_{E_1,E_2}}"{description}, from=1-5, to=3-6]
	\arrow["{1_{E_1\cup E_2}}"{description}, from=1-5, to=3-4]
	\arrow["s"{description}, from=3-6, to=3-4]
\end{tikzcd}\]
Because $(Y,(q_f)_f)$ is a cone over the diagram $D_\bfomega$, we also have the following commutative triangles:
\[\begin{tikzcd}
	& {Y} &&& {Y} \\
	&&& {} \\
	{\M_r^f(\textbf{2},p_{E_2})} & {\M_r^f(\textbf{3},p_{E_1,E_2})} & {\M_r^f(\textbf{2},p_{E_1})} & {\M_r^f(\textbf2,p_{E_1\cup E_2})} && {\M_r^f(\textbf{3},p_{E_1,E_2})}
	\arrow["{q_{1_{E_1,E_2}}}"{description}, from=1-2, to=3-2]
	\arrow["{q_{1_{E_1}}}"{description}, from=1-2, to=3-3]
	\arrow["{\M_r^f(s_1)}"', from=3-2, to=3-3]
	\arrow["{q_{1_{E_2}}}"{description}, from=1-2, to=3-1]
	\arrow["{\M_r^f(s_2)}", from=3-2, to=3-1]
	\arrow["{q_{1_{E_1,E_2}}}"{description}, from=1-5, to=3-6]
	\arrow["{q_{1_{E_1\cup E_2}}}"{description}, from=1-5, to=3-4]
	\arrow["\M_r^f(s)"{description}, from=3-6, to=3-4]
\end{tikzcd}\]
Using the above diagrams we find \[\mu_y(E_1\cup E_2)=q_{1_{E_1\cup E_2}}(y)_1=q_{1_{E_1,E_2}}(y)_1+q_{1_{E_1,E_2}}(y)_2= q_{1_{E_1}}(y)_1+q_{1_{E_2}}(y)_1= \mu_y(E_1)+\mu_y(E_2),\]
which shows that $\mu_y$ is finitely additive. 

Let $(E_n)_n$ be a sequence of measurable subsets of $\bfomega$ that decreases to $\emptyset$. Because $0\leq \mu_y(E_n)\leq r\mathbb{P}(E_n)$, also \[0\leq \lim_n\mu_y(E_n)\leq \lim_n\mathbb{P}(E_n)=0.\] We can conclude that $\mu_y$ is an element of $\M_r(\bfomega)$. 

The assignment $y\mapsto \mu_y$ defines a map $q:Y\to \M_r(\bfomega)$. It can be checked that this map is $1$-Lipschitz and that it defines morphism of cones. Furthermore, this morphism of cones is unique. We can now conclude that $\M_r(\bfomega)= \mathrm{Ran}_i\M_r(\bfomega)$.
\end{proof}

We have just described how the functor $\mathrm{Ran}_i\M_r^f$ behaves on objects. In the following proposition we will study how it acts on morphisms.

\begin{prop}\label{RN2.7}
For a measure preserving map of probability space $f:\bfomega_1\to \bfomega_2$,  
\[\M_r(f)(\mu)=\mu\circ f^{-1},\]
for all $\mu\in \M_r(\bfomega_1).$
\end{prop}
\begin{proof}
By the universal property of right Kan extensions, we know that $M_r(f):M_r(\bfomega_1)\to M_r(\bfomega_2)$ is the unique morphism such that \[\begin{tikzcd}
	{\mathrm{M}_r(\mathbf{\Omega}_1)} && {\mathrm{M}_r(\mathbf{\Omega}_2)} \\
	& {\mathrm{M}_r^f(A,p)}
	\arrow["{\mathrm{M}_r(f)}", from=1-1, to=1-3]
	\arrow["{p_{hf}^1}"', from=1-1, to=2-2]
	\arrow["{p_f^2}", from=1-3, to=2-2]
\end{tikzcd}\] commutes for every measure-preserving map $h:\bfomega_2\to (A,p)$, where $(A,p)$ is some finite probability space. Here $p_{hf}^1$ and $p_f^2$ are the projection maps defined in the proof of Theorem \ref{RN2.6}. It is clear that the map $\M_r(\bfomega_1)\to \M_r(\bfomega_2)$ defined by the assignment $\mu\mapsto \mu\circ f^{-1}$ satisfies this property and therefore the claim follows.
\end{proof}

Theorem \ref{RN2.6} and Proposition \ref{RN2.7} tell us that the functor $\mathrm{Ran}_i\M_r^f:\textbf{Prob}\to \mathbf{CMet}$ expresses measures. We will therefore use the notation $\M_r:=\mathrm{Ran}_i\M_r^f$ from now on.

Furthermore, note that in the proofs of Theorem \ref{RN2.6} and Proposition \ref{RN2.7} we have not used any non-trivial measure-theoretic results. The proofs in this section are straightforward categorical proofs.

For $r\leq s$, there is a natural transformation $\M_r^f\to \M_s^f$ and therefore a natural transformation $\M_r\to \M_s$. This natural transformation is given by the inclusion maps $$\M_r(\bfomega)\to \M_s(\bfomega),$$ for all probability spaces $\bfomega$. This gives us a diagram $$D_\mathrm{M}:(0,\infty)\to [\textbf{Prob},\mathbf{CMet}]$$ of functors and natural transformations \[\begin{tikzcd}
	\ldots & {M_1} & \ldots & {M_2} & \ldots & {M_r} & \ldots.
	\arrow[from=1-1, to=1-2]
	\arrow[from=1-2, to=1-3]
	\arrow[from=1-3, to=1-4]
	\arrow[from=1-4, to=1-5]
 \arrow[from=1-5, to=1-6]
  \arrow[from=1-6, to=1-7]
\end{tikzcd}\]

In the rest of this section, we will study what the colimit of this diagram looks like. The next proposition tells us how the colimiting functor acts on objects.

\begin{prop}\label{RN2.8}
    Let $\bfomega$ be a probability space, then $\mathrm{colim}D_\mathrm{M}(\bfomega)=\M(\bfomega).$
\end{prop}
\begin{proof}
Consider the subset $$S:=\{\mu\in \M(\bfomega)\mid \exists r>0: \mu \leq r\mathbb{P}\}.$$ We will show that $S$ is dense in $\M(\bfomega)$. For $\mu\in \M(\bfomega)$, define $$\mu_n:=\mu\wedge n\mathbb{P},$$
This means that $$\mu_n(E)=\sup\left\{\sum_{k=1}^\infty \mu(E_k)\wedge n\mathbb{P}(E_k)\mid \bigcup_{k=1}^\infty E_k \subseteq E\right\},$$
and clearly $\mu_n\leq n\mathbb{P}$ and $\mu_n\leq \mu$.

For a countable partition $(E_k)_{k=1}^\infty$ and a natural number $n\geq 1$, we find $$\sum_{k=1}^\infty\lvert \mu(E_k)-\mu_n(E_k)\rvert = \sum_{k=1}^\infty \mu(E_k)-\mu_n(E_k) = \mu(\Omega)-\mu(\Omega)\wedge n.$$
Taking the supremum over all countable partitions of $\Omega$ gives $$d_{\mathrm{TV}}(\mu,\mu_n)\leq \mu(\Omega)-(\mu(\Omega)\wedge n)$$ and therefore $\mu_n\to \mu$ in $\M(\bfomega)$. 

Thus, for any complete metric space $Y$ and $1$-Lipschitz map $f:S\to Y$, there is a unique $1$-Lipschitz map $\tilde{f}:\M(\bfomega)\to Y$. It follows that $\mathrm{colim}(D_\mathrm{M}(\bfomega))=\M(\bfomega)$ and therefore $(\mathrm{colim}D_\mathrm{M})(\bfomega)=\M(\bfomega)$. 
\end{proof}

We will now discuss what the colimit of $D_\mathrm{M}$ does on morphisms.

\begin{prop}\label{RN2.9}
    Let $\bfomega_1:=(\Omega_1,\mathcal{F}_1,\mathbb{P}_1)$ and $\bfomega_2:=(\Omega_2,\mathcal{F}_2,\mathbb{P}_2)$ be probability spaces and let $f:\bfomega_1\to\bfomega_2$ be a measure preserving map. Then $(\mathrm{colim}D_\mathrm{M})(f)$ is the $1$-Lipschitz map $\M(\bfomega_1)\to\M(\bfomega_2)$ defined by $$\mu\mapsto \mu\circ f^{-1}.$$
\end{prop}
\begin{proof}
    The map $(\mathrm{colim}D_\mathrm{M})(f)$ is the unique map $\M(\bfomega_1)\to \M(\bfomega_2)$ such that the following diagram commutes for every natural number $r>0$

    \[\begin{tikzcd}
	{\mathrm{M}(\mathbf{\Omega}_1)} && {\mathrm{M}(\mathbf{\Omega}_2)} \\
	{\mathrm{M}_r(\mathbf{\Omega}_1)} && {\mathrm{M}_r(\mathbf{\Omega}_2)}
	\arrow[from=2-1, to=1-1]
	\arrow[from=2-3, to=1-3]
	\arrow["{\mathrm{M}_r(f)}"', from=2-1, to=2-3]
	\arrow["{(\mathrm{colim}D_\mathrm{M})(f)}", from=1-1, to=1-3]
\end{tikzcd}\]

The $1$-Lipschitz maps $\M(\bfomega_1)\to \M(\bfomega_2):\mu\mapsto \mu\circ f^{-1}$ satisfies this condition and therefore it has to be equal to $(\mathrm{colim}D_\mathrm{M})(f)$.
\end{proof}

Proposition \ref{RN2.8} and Proposition \ref{RN2.9} tell us that the functor $\mathrm{colim}D_\M$ describes measures. Therefore we will from now on use the notation $$\M:=\mathrm{colim}D_\M.$$

\subsection{\texorpdfstring{The random variables functor $\RV$}{The random variables functor RV}}\label{RNSection2.3}

We will start this section by describing what the right Kan extension of the functor $\RV_r^f:\textbf{Prob}_f\to\mathbf{CMet}$ along the functor $i:\textbf{Prob}_f\to \textbf{Prob}$ looks like. 
\[\begin{tikzcd}
	{\textbf{Prob}_f} && {\mathbf{CMet}} \\
	\\
	{\textbf{Prob}}
	\arrow["{\mathrm{RV}_r^f}", from=1-1, to=1-3]
	\arrow["i"', from=1-1, to=3-1]
	\arrow["{\mathrm{Ran}_i\mathrm{RV}_r^f}"', dashed, from=3-1, to=1-3]
\end{tikzcd}\]
We will do this by first showing how $\mathrm{Ran}_i\RV_r^f$ acts on objects in Theorem \ref{RN2.10} and then how it acts on morphisms in Proposition \ref{RN2.11}. The conclusion will be that this right Kan extension describes bounded random variables on arbitrary probability spaces. We will therefore introduce the notation $\RV_r$ to mean the functor $\mathrm{Ran}_i\RV_r^f:\textbf{Prob}\to \mathbf{CMet}$.

We will proceed to section by showing that these functors form a diagram $D_\RV$: 
\[\begin{tikzcd}
	\ldots & {\mathrm{RV}_1} & \ldots &  {\mathrm{RV}_2} & \ldots & {\mathrm{RV}_n} & \ldots
	\arrow[from=1-1, to=1-2]
	\arrow[from=1-2, to=1-3]
	\arrow[from=1-3, to=1-4]
	\arrow[from=1-4, to=1-5]
 \arrow[from=1-5, to=1-6]
  \arrow[from=1-6, to=1-7]
\end{tikzcd}\] In the remaining part of the section we will study the colimit of $D_\RV$. We will show that this colimiting functor describes random variables and we will therefore denote this functor as $\RV$.

\begin{thm}\label{RN2.10}
Let $\bfomega:=(\Omega,\mathcal{F},\mathbb{P})$ be a probability space. Then
\[\textup{Ran}_i\RV_r^f(\bfomega)=\RV_r(\bfomega).\]
\end{thm}
\begin{proof}
Let $U:\bfomega\downarrow i \to \textbf{Prob}_f$ be the forgetful functor and let $D_\bfomega$ denote the diagram 
\[\bfomega\downarrow i \xrightarrow{U} \textbf{Prob}_f\xrightarrow{\RV_r^f}\mathbf{CMet}.\]
We will now show that $\RV_r(\bfomega)=\lim D_\bfomega.$

For a measure-preserving map $f:\bfomega\to (A,p)$ where $(A,p)$ is some finite probability space, define a map $p_f:\RV_r(\bfomega)\to \RV_r^f(A,p)$ by \[p_f(g)(a):=\begin{cases}\frac{1}{p_a}\int_{f^{-1}(a)}g\text{d}\mathbb{P} \text{ if }p_a\not=0\\
0 \text{ otherwise},\end{cases}\]
for every $g$ in $\RV_r(\bfomega)$ and $a$ in $A$. It is straightforward to check that the definition of $p_f$ is independent of the choice of representative. It can also be checked that $p_f$ is $1$-Lipschitz.
Consider a commutative diagram \[\begin{tikzcd}\bfomega \arrow[r,"f_1"]\arrow[rd,"f_2",swap] & (A,p)\arrow[d,"s"] \\
& (B,q)\end{tikzcd}\]
Let $g$ in $\RV_r(\bfomega)$ and $b$ in $B$, we have \begin{align*}
    \left[\left(\RV_r^f(s)\circ p_{f_1}(g)\right)(b) \right]q_b &= \sum_{a\in s^{-1}(b)}p_ap_{f_1}(g)(a) \\
    &= \sum_{a\in s^{-1}(b)}\int_{f^{-1}(a)}g\text{d}\mathbb{P}\\
    &= \int_{s^{-1}(b)}g\text{d}\mathbb{P}.
\end{align*}
It now follows that $(\RV_r(\bfomega),(p_f)_f)$ is a cone over the diagram $D_\bfomega$. We will now show that this cone is universal. To do that, we consider another cone $(Y,(q_f)_f)$ over the diagram $D_\bfomega$. 

For $y\in Y$ and measure-preserving map $\bfomega\xrightarrow{f} (A,p)$, we define a simple function $\Omega\to [0,\infty)$ as follows: 
\[s_f^y=\sum_{a\in A}q_f(y)(a)1_{f^{-1}(a)}.\]
Consider a commutative diagram \[\begin{tikzcd}
	\bfomega && {(A,p)} \\
	&& {(B,r)}
	\arrow["f", from=1-1, to=1-3]
	\arrow["g"', from=1-1, to=2-3]
	\arrow["s", from=1-3, to=2-3]
\end{tikzcd}\]
We find for every $b\in B$ that 
\[r_l \int_{g^{-1}(b)}s_f^y\text{d}\mathbb{P}=\sum_{s(a)=b}q_f(y)(a)p_a=q_g(y)(b)r_b\]
It follows that $p_g(s_f^y)=q_g(y)$.
Note that $\bfomega\downarrow i$ is cofiltered and therefore $(s_f^y)_f$ forms a net in $\RV_r(\bfomega)$. Suppose now that $(s_f^y)_f$ has a limit $s^y$ in $\RV_r(\bfomega)$. Then it is easy to see that $p_g(s^y)=q_g(y)$ for all $g \in \bfomega\downarrow i$. By the Riesz-Fischer theorem (Theorem \ref{RN2.1}) we only need to show that $(s_f^y)_f$ is a Cauchy net.

For this we will use the following two inequalities, which we will prove in the Appendix. For a commutative diagram \[\begin{tikzcd}
	\bfomega && {(A,p)} \\
	&& {(B,r)}
	\arrow["f", from=1-1, to=1-3]
	\arrow["g"', from=1-1, to=2-3]
	\arrow["s", from=1-3, to=2-3]
\end{tikzcd}\]we have  \begin{equation}\label{1}
    \mathbb{E}[\left(s_g^y\right)^2] \leq \mathbb{E}[(s_f^y)^2]. 
\end{equation}
and 
\begin{equation}\label{2} 0 \leq d_{L^1}(s_f^y,s_g^y)^2\leq \mathbb{E}[(s_f^y-s_g^y)^2]= \mathbb{E}[(s_f^y)^2]-\mathbb{E}[(s_g^y)^2]\end{equation}
Now by inequaltiy \eqref{1} we conclude that $\left(\mathbb{E}\left[(s_f^y)^2\right]\right)_f$ is a bounded, monotone net and therefore it converges. Using \eqref{2} we can now conclude that $(s_f^y)_f$ is a Cauchy net. The Riesz-Fischer theorem (Theorem \ref{RN2.1}) tells us that the net $(s_f^y)_f$ converges to some $s^y$ in $\RV_r(\bfomega)$. This defines a map $Y\to G_n(X,\Sigma,\mathbb{P})$. It can be checked that this is a $1$-Lipschitz map. Because $p_g(s^y)=q_g(y)$ for every $g$ in $\bfomega\downarrow i$ it follows that this map is in fact a morphism of cones. Finally, it is straightforward to check that this is the unique morphism of cones. We can now conclude that $\RV_r(\bfomega)= \mathrm{Ran}_i\RV_r^f (\bfomega)$.
\end{proof}

We now know what the right Kan extension $\mathrm{Ran}_i\RV_r^f$ does on objects, but not yet how it acts on morphisms. This is described in the following proposition.

\begin{prop} \label{RN2.11}
Let  $f:\bfomega_1 \to \bfomega_2$ be a measure preserving map of probability spaces. Let $g\in \RV_r(\bfomega_1)$. Then $$\RV_r(f)(g)=\mathbb{E}[g\mid f].$$
\end{prop}
\begin{proof}
By the defining property of conditional expectation it is enough to show that 
\[\mathbb{E}_{\bfomega_1}[g1_{f^{-1}(B)}]=\mathbb{E}_{\bfomega_2}[\RV_r(f)(g)1_B]. \]
for all measurable subsets $B$ of $\bfomega_2$.
Since $\RV_r$ is defined as $\mathrm{Ran}_i\RV_r^f$, $\RV_r(f)$ is the unique map $\RV_r(X)\to \RV_r(Y)$ such that the following diagram commutes for every measure preserving map $h:\bfomega_2\to (A,p)$, where $(A,p)$ is some finite probability space.
\[\begin{tikzcd}
	{\RV_r(\bfomega_1)} && {\RV_r(\bfomega_2)} \\
	& {\RV_r^f(A,p)} & 
	\arrow["{\RV_r(f)}", from=1-1, to=1-3]
	\arrow["{p_{hf}^1}"', from=1-1, to=2-2]
	\arrow["{p_h^2}", from=1-3, to=2-2]
\end{tikzcd}\]
Here $p_{hf}^1$ and $p_f^2$ are the projection maps defined in the proof of Theorem \ref{RN2.10}. Let $\textbf{2}:=\{0,1\}$ and let $r$ be the probability measure on $\textbf{2}$ defined by $r_0:=\mathbb{P}_Y(B)$.
Consider the measure preserving map \[h:\bfomega_2\to (\textbf{2},r)\] defined by the assignment \[h(\omega):=\begin{cases} 1 \text{ if }\omega\in B\\
0 \text{ otherwise.}\end{cases}\]
We now find 

\[\mathbb{E}_{\bfomega_1}[g1_{f^{-1}(B)}]=p_{hf}^{\bfomega_1}(g)= p_h^{\bfomega_2}(\RV_r(f)(g)) = \mathbb{E}_{\bfomega_2}[\RV_r(f)(g)1_B].\] This proves the claim.

\end{proof}

\begin{rem}\label{RN2.12}
Note that the net $(s^y_f)_f$ in the proof of Theorem \ref{RN2.10} could be interpreted as a martingale. The argument we used to show that this net convergence then corresponds to the proof of the martingale convergence theorem in \cite{lamb}.
\end{rem}
For positive real numbers $r\leq s$, there is a natural transformation $\RV_r^f\to \RV_s^f$ and therefore a natural transformation $\RV_r\to \RV_s$. This natural transformation is given by the inclusion maps $$\RV_r(\bfomega)\to \RV_s(\bfomega),$$ for all probability space $\bfomega$. This gives a diagram $D_\RV:(0,\infty)\to [\textbf{Prob},\mathbf{CMet}]$ of functors an natural transformations \[\begin{tikzcd}
	\ldots & {\mathrm{RV}_1} & \ldots & {\mathrm{RV}_2} & \ldots & {\mathrm{RV}_r} & \ldots.
	\arrow[from=1-1, to=1-2]
	\arrow[from=1-2, to=1-3]
	\arrow[from=1-3, to=1-4]
	\arrow[from=1-4, to=1-5]
 	\arrow[from=1-5, to=1-6]
   	\arrow[from=1-6, to=1-7]
\end{tikzcd}\] In the rest of this section, we will describe what the colimit of this diagram looks like. We will describe the colimiting functor's behaviour on objects and morphisms in the following two propositions.

\begin{prop}\label{RN2.13}
    Let $\bfomega$ be a probability space, the $(\mathrm{colim}D_\RV)(\bfomega)=\RV(\Omega).$
\end{prop}

\begin{proof}
    Consider the subset $$S:=\{f\in \RV(\bfomega)\mid \exists r>0: \mathbb{P}(f\leq r)=1\}.$$ We will show that $S$ is dense in $\RV(\Omega)$. For $f\in \RV(\Omega)$, define $$f_n:=f\wedge n.$$
    Clearly, $f_n\in S$ and by the Monotone Convergence Theorem, we have that $f_n\to f$ in $\RV(\bfomega)$.

    It follows now that for any complete metric space $Y$ and $1$-Lipschitz map $f:S\to Y$, there is a unique $1$-Lipschitz map $\tilde{f}:\RV(\bfomega)\to Y$. From this it follows that $\mathrm{colim}(D_\RV(\bfomega))=\RV(\bfomega)$ and thus $(\mathrm{colim}D_\RV)(\bfomega)=\RV(\bfomega).$
\end{proof}

We will end this section by showing how $\mathrm{colim}D_\RV$ acts on morphisms.

\begin{prop}\label{RN2.14}
    Let $\bfomega_1:=(\Omega_1,\mathcal{F}_1,\mathbb{P}_1)$ and $\bfomega_2:=(\Omega_2,\mathcal{F}_2,\mathbb{P}_2)$ be probability spaces and let $g:\bfomega_1\to \bfomega_2$ be a measure preserving map. Then $(\mathrm{colim}D_\RV)(g)$ is the $1$-Lipschitz map $\RV(\bfomega_1)\to \RV(\bfomega_2)$ defined by $$f\mapsto \mathbb{E}[f\mid g].$$
\end{prop}

\begin{proof}
    The map $(\mathrm{colim}D_\RV)(g)$ is the unique map $\RV(\bfomega_1)\to \RV(\bfomega_2)$ such that the following diagram commutes for every natural number $n\geq 1$: \[\begin{tikzcd}
	{\mathrm{RV}(\mathbf{\Omega}_1)} && {\mathrm{RV}(\mathbf{\Omega}_1)} \\
	{\mathrm{RV}_n(\mathbf{\Omega}_1)} && {\mathrm{RV}_n(\mathbf{\Omega}_1)}
	\arrow["{\mathrm{RV}_n(g)}"', from=2-1, to=2-3]
	\arrow["{(\mathrm{colim}D_\mathrm{RV})(g)}", from=1-1, to=1-3]
	\arrow[from=2-1, to=1-1]
	\arrow[from=2-3, to=1-3]
\end{tikzcd}\]
From Proposition \ref{RN2.11} it follows that the $1$-Lipschitz map $\RV(\bfomega_1)\to \RV(\bfomega_2):f\mapsto \mathbb{E}[f\mid g]$ satisfies this condition and therefore it has the be equal to $(\mathrm{colim}D_\RV)(g)$.
\end{proof}

Proposition \ref{RN2.13} and Proposition \ref{RN2.14} tell us that the functor $\mathrm{colim}D_\RV$ describes random variables. Therefore we will from now on use the notation $$\RV:=\mathrm{colim}D_\RV.$$
\subsection{The Radon-Nikodym theorem}\label{RNSection2.4}

We will now conclude section \ref{RNSection2} by giving a categorical proof of the Radon-Nikodym theorem. We will first look at a weaker bounded version (Theorem \ref{RN2.16}) and then extend this to the general version (Theorem \ref{RN2.18}). In Proposition \ref{RN2.15} and Proposition \ref{RN2.17}, we give the concrete construction of the correspondence between random variables and measures that we obtain from the categorical proofs.


For the weaker bounded version of the Radon-Nikodym theorem, we will use the functors $\M_r$ and $\RV_r$ as defined in section \ref{RNSection2.2} and section \ref{RNSection2.3}. 

Recall that we have a natural transformation $\rho_r^f:\RV_r^f\to \M_r^f$. This induces a natural transformation $\mathrm{Ran}_i\rho_r^f:\mathrm{Ran}_i\RV_r^f\to \mathrm{Ran}_i\M_r^f$. This is a natural transformation $\RV_r\to \M_r$, which we will denote by $\rho_r$. In the following proposition, we will describe this natural transformation.

\begin{prop}\label{RN2.15}
Let $\bfomega$ be a probability space. For $g\in \RV_r(\bfomega)$ and $B$ a measurable subset of $\bfomega$, then
\[(\rho_r)_\bfomega(g)(B)=\mathbb{E}[g1_B].\]
\end{prop}
\begin{proof}
Since $\rho_r$ is defined as $\mathrm{Ran}_i\rho_r^f$, the map $(\rho_r)_\bfomega$ is the unique map $\RV_r(\bfomega)\to \M_r(\bfomega)$ such that the following diagram commutes for every measure preserving map $h:\bfomega\to (A,p)$, where $(A,p)$ is some finite probability space.
\[\begin{tikzcd}
	{\RV_r(\bfomega)} && {\M_r(\bfomega)} \\
	{\RV_r^f(A,p)} && {\M_r^f(A,p)}
	\arrow["{(\rho_r)_\bfomega}"{description}, from=1-1, to=1-3]
	\arrow["{p^\RV_h}"', from=1-1, to=2-1]
	\arrow["{p_h^\M}", from=1-3, to=2-3]
	\arrow["{\left(\rho_r^f\right)_A}"{description}, from=2-1, to=2-3]
\end{tikzcd}\]
Here $p^\RV_h$ and $p^\M_h$ are the projection maps defined in the proofs of Theorem \ref{RN2.6} and Theorem \ref{RN2.10}.

Let $\textbf{2}:=\{0,1\}$ and let $r$ be the probability measure on $\textbf{2}$ defined by $r_1:=\mathbb{P}(B)$. We have a measure preserving map \[h:\bfomega\to (\textbf{2},r)\]
defined by the assignment \[h(x):=\begin{cases} 1 \text{ if }x\in B\\ 0 \text{ otherwise}. \end{cases}\]
The commutative diagram for this measure preserving map gives us 
\[(\rho_r)_\bfomega(g)(B)= p_h^\M\circ (\rho_r)_\bfomega(g)_1 = (\rho_r^f)_A\circ p_h^\RV(g)_1 = \mathbb{E}[g1_{h^{-1}(1)}]= \mathbb{E}[g1_B].\]
\end{proof}
Because $\rho_r^f$ is an isomorphism, so is $\rho_r$. This gives us the bounded Radon-Nikodym theorem. 
\begin{thm}[Bounded Radon-Nikodym]\label{RN2.16}
The natural transformation $\rho_r:\RV_r\to \M_r$ is an isomorphism.
\end{thm}
\begin{proof}
By Proposition \ref{RN2.5}, we know that $\RV_r^f\xrightarrow{\rho_r^f}\M_r^f$ is an isomorphism. Using Theorem \ref{RN2.6} and Theorem \ref{RN2.10} and the fact that the Kan extension is functorial, we conclude that \[\RV_r \xrightarrow{\text{Ran}_i\rho_r^f}\M_r\] is a natural isomorphism.
\[\begin{tikzcd}
	{\textbf{Prob}_f} && {\mathbf{CMet}} \\
	\\
	& {\textbf{Prob}}
	\arrow[""{name=0, anchor=center, inner sep=0}, "{\M_r^f}", curve={height=-6pt}, from=1-1, to=1-3]
	\arrow[""{name=1, anchor=center, inner sep=0}, "{\RV_r^f}"', curve={height=6pt}, from=1-1, to=1-3]
	\arrow["i"{description}, from=1-1, to=3-2]
	\arrow[""{name=2, anchor=center, inner sep=0}, "{\M_r}", curve={height=-6pt}, from=3-2, to=1-3]
	\arrow[""{name=3, anchor=center, inner sep=0}, "{\RV_r}"', curve={height=6pt}, from=3-2, to=1-3]
	\arrow["\simeq"{description}, Rightarrow, draw=none, from=1, to=0]
	\arrow["\simeq"{description}, Rightarrow, draw=none, from=2, to=3]
\end{tikzcd}\]
\end{proof}

The natural transformations $\rho_r:\RV_r\to \M_r$ for every $r>0$, induce a morphism of diagrams $\tilde{\rho}:D_\RV\to D_\M$. Therefore we obtain a natural transformation $\mathrm{colim}\Tilde{\rho}:\mathrm{colim}D_\RV\to \mathrm{colim}D_\M$. This is a natural transformation $\RV\to \M$, which we will denote by $\rho$. The following proposition describes this natural transformation.

\begin{prop}\label{RN2.17}
Let $\bfomega$ be a probability space, then $\rho_\bfomega$ is the map $\RV(\bfomega)\to \M(\bfomega)$ defined by the assignment \[f\mapsto \int_{(-)}f\text{d}\mathbb{P}.\]
\end{prop}
\begin{proof}
The map $\rho_\bfomega$ is the unique map that makes the following diagram commute for every $r>0$.
\[\begin{tikzcd}
	{\RV(\bfomega)} && {\M(\bfomega)} \\
	{\RV_r(\bfomega)} && {\M_r(\bfomega)}
	\arrow["{(\rho_r)_\bfomega}", from=2-1, to=2-3]
	\arrow["{(\text{colim}\rho)_\bfomega}", from=1-1, to=1-3]
	\arrow["{i_{r,\bfomega}}", from=2-1, to=1-1]
	\arrow["{j_{r,\bfomega}}"', from=2-3, to=1-3]
\end{tikzcd}\]
Where $i_{r,\bfomega}$ and $j_{r,\bfomega}$ are the inclusion maps. We have that for all $r>0$, \[\int_{(-)}i_{r,\bfomega}(f)\text{d}\mathbb{P}=j_{r,\bfomega}\rho_r(f)\]
for all $f\in \RV_r(\bfomega)$. The claim now follows.
\end{proof}

We are now ready to complete the categorical proof for the Radon-Nikodym theorem.
\begin{thm}[Radon-Nikodym]\label{RN2.18}
The natural transformation $\rho:\RV\to \M$ is an isomorhpism.
\end{thm}

\begin{proof}
Because $\rho_r:\RV_r\to \M_r$ is an isomorphism for every $r>0$, so is $\tilde{\rho}:D_\RV\to D_\M$. We find that $\rho:=\textup{colim}\tilde{\rho}:\RV\to \M$ is a natural isomorphism.

\[\begin{tikzcd}
	\ldots & {\RV_1} & {\ldots} & {\RV_3} & \cdots & \RV \\
	\ldots & {\M_1} & {\ldots} & {\M_3} & \cdots & \M
	\arrow[from=1-1, to=1-2]
	\arrow[from=1-2, to=1-3]
	\arrow[from=1-3, to=1-4]
	\arrow[from=1-4, to=1-5]
 \arrow[from=1-5, to=1-6]
	\arrow[from=2-1, to=2-2]
	\arrow[from=2-2, to=2-3]
	\arrow[from=2-3, to=2-4]
	\arrow[from=2-4, to=2-5]
 \arrow[from=2-5, to=2-6]
	\arrow["\simeq"{description}, from=1-2, to=2-2]
	\arrow["\simeq"{description}, from=1-4, to=2-4]
	\arrow["\simeq"{description}, from=1-6, to=2-6]
\end{tikzcd}\]
\end{proof}

\section{The martingale convergence theorem} \label{RNSection3}
In this section we will focus on a special class of stochastic processes, namely \emph{martingales}. These stochastic processes have nice convergence properties, of which we will prove one categorically later in this section. Important examples of martingales are Brownian motion and unbiased random walks.

Let $\mathbf{\Omega}:=(\Omega,\mathcal{F},\mathbb{P})$ be a probability space and let $I$ be a directed poset. A \textbf{filtration} is an indexed collection $(\mathcal{F}_i)_{i\in I}$ of $\sigma$-subalgebras of $\mathcal{F}$ such that $\mathcal{F}_i\subseteq \mathcal{F}_j$ for $i\leq j$ and such that $$\sigma\left(\bigcup_{i\in I}\mathcal{F}_i\right)=\mathcal{F}.$$  We say that $(\Omega,\mathcal{F},\left( \mathcal{F}_i\right)_{i\in I}, \mathbb{P})$ is a \textbf{filtered probability space}. The probability space $(\Omega,\mathcal{F}_i,\mathbb{P}\mid_{\mathcal{F}_i})$ is denoted by $\mathbf{\Omega}_i$. For $i\leq j$ in $I$, there is a measure preserving map $f_{ij}:\bfomega_j\to \bfomega_j$ and for every $i\in I$  there is a measure-preserving map $f_i:\mathbf{\Omega}\to \mathbf{\Omega}_i$.

An indexed collection $(X_i)_{i\in I}$ of random variables such that $X_i\in \RV(\bfomega_i)$ is called a \textbf{martingale} if $$\mathbb{E}[X_j\mid f_{ij}] = X_i$$ for all $i\leq j$ in $I$. 

Martingales often have nice convergence properties. We will categorically prove a weaker version of the following \emph{martingale convergence theorems} in section \ref{RNSection3.4}.

\begin{thm}[Doob's $L^1$ martingale convergence theorem]\label{RN3.1}
Let $(X_n)_{n=1}^\infty$ be a martingale such that $$\lim_{\lambda \to \infty}\sup_n\mathbb{E}[X_n1_{\{X_n>\lambda \}}] = 0,$$
then $(X_n)_n$ converges to a random variable $X$ in $L^1$-norm and  for all $n\geq 1$, $$\mathbb{E}[X\mid f_n]=X_n.$$
\end{thm}

\begin{thm}[Doob's $L^p$ martingale convergence theorem]\label{RN3.2}
Let $p>1$ and let $(X_n)_{n=1}^\infty$ be a martingale such that $$\sup_n\mathbb{E}[X_n^p]<\infty,$$
then $(X_n)_n$ converges to a random variable $X$ in $L^p$-norm and  for all $n\geq 1$, $$\mathbb{E}[X\mid f_n]=X_n.$$
\end{thm}

To give a categorical proof, the setting from section \ref{RNSection2} does not quite work. We need to change everything from section \ref{RNSection2} to the \emph{enriched} setting. We will enrich everything over the closed monoidal category $\mathbf{CMet}$, which we will discuss in section \ref{RNSection3.1}. We then show in section \ref{RNSection3.2} and section \ref{RNSection3.3} that the results from section \ref{RNSection2} still work when everything is enriched over $\mathbf{CMet}$. We then conclude section \ref{RNSection3} by giving a categorical proof for a weaker version of the martingale convergence theorems in section \ref{RNSection3.4}.

\subsection{\texorpdfstring{The closed monoidal category $\mathbf{CMet}$}{The closed monoidal category CMet}}\label{RNSection3.1}
In this section we will give an overview of well-known results about metric spaces. For completeness, we give proofs for all the results.

Let $i:\mathbf{CMet}\to \textbf{Met}$ be the inclusion functor of the full subcategory of complete metric spaces in the category of metric spaces.

\begin{prop}\label{RN3.3}
The category $\mathbf{CMet}$ is complete and $i:\mathbf{CMet}\to \mathbf{Met}$ preserves these limits. 
\end{prop}
\begin{proof}
    For a collection of complete metric spaces $(X_i,d_i)_{i\in I}$ let $X:=\prod_{i\in I} X_i$ and define $d:X\times X \to [0,\infty]$ by \[d((x_i)_i,(y_i)_i):=\sup_{i\in I}d_i(x_i,y_i).\]
    It is clear that $d$ defines a metric\footnote{Recall that all our metrics are really \emph{extended pseudometrics}. Here, $d$ can take the value $\infty$ even when $d_i$ is finite for every $i$ in $I$.} and that the projection maps $\pi_i:X\to X_i$ are $1$-Lipschitz. Let $(x^n)_n$ be a cauchy sequence in $(X,d)$. Clearly, $(x^n_i)_n$ is a Cauchy sequence in $(X_i,d_i)$ for all $i\in I$. It follows that $(x^n_i)_n$ converges to an element $x_i$ in $(X_i,d_i)$. Denote $x:=(x_i)_i$. For $\epsilon>0$, there exists an $N\geq 1$ such that for $n_1,n_2\geq N$, \[d((x^{n_1}_i)_i,(x^{n_2}_i)_i)\leq \epsilon.\]
    For $i\in I$, there exists $M_i\geq N$ such that $d_i(x_i^{M_i},x_i)\leq \epsilon$. It follows now that \[d_i(x_i^N,x_i)\leq d(x_i^N,x_i^{M_i})+d(x_i^{M_i},x_i)\leq 2\epsilon.\]
    Since $N$ does not depend on $i$, we can take the supremum over all $i\in I$ and conclude that \[d(x^N,x)\leq 2\epsilon.\]
    Therefore, $(x^n)_n$ converges to $x$ in $(X,d)$ and thus it is a complete metric space. The complete metric space $(X,d)$ is the product of $(X_i,d_i)_{i\in I}$.

    For morphisms $f,g:(X,d_X)\to (Y,d_Y)$ in $\mathbf{CMet}$. Let \[E:=\{x\in X\mid f(x)=g(x)\}\] and let $d_E$ be the restriction of $d_X$ to $E\times E$. This forms a metric space $(E,d_E)$. For a Cauchy sequence $(e_n)_n$ in $(E,d_E)$ we know that $(e_n)_n$ converges to some $x$ in $(X,d_X)$. Because $f$ and $g$ are $1$-Lipschitz, we see that \[f(x)=\lim_{n}f(e_n)=\lim_ng(e_n)=g(x).\]
    Therefore, $x\in E$ and $(E,d)$ is complete. The complete metric space $(E,d)$ is the equalizer of $f$ and $g$ in $\mathbf{CMet}$.

    It is clear that $i:\mathbf{CMet}\to \textbf{Met}$ preserves these limits.
\end{proof}
\begin{prop}[Completion]\label{RN3.4}
    The inclusion $i:\mathbf{CMet}\to \mathbf{Met}$ has a left adjoint.
\end{prop}
\begin{proof}
The completion functor $\overline{(-)}:\mathbf{Met}\to \mathbf{CMet}$ that sends a metric space $(X,d)$ to its completion $\overline{(X,d)}$ is left adjoint to the inclusion functor $i:\mathbf{CMet}\to \mathbf{Met}$.
\end{proof}

Proposition \ref{RN3.4} tells us that $\mathbf{CMet}$ is a \emph{reflective} subcategory of $\mathbf{Met}$. We will use this in the following result about colimits in $\mathbf{Met}$ and $\mathbf{CMet}$.

\begin{prop}\label{RN3.5}
    The categories $\mathbf{CMet}$ and $\mathbf{Met}$ are cocomplete.
\end{prop}
\begin{proof}
For a collection of metric space $(X_i,d_i)_{i\in I}$, let $X:=\coprod_{i\in I}X_i$ and define $d:X\times X\to [0,\infty]$ by \[d(x,y):=\begin{cases}d_i(x,y) \text{ if } x,y\in X_i\\ \infty \text{ otherwise}.\end{cases}\]
Then $(X,d)$ forms a metric spaces and the inclusion maps $\iota_i:X_i\to X$ are $1$-Lipschitz maps. The metric space $(X,d)$ is the coproduct of $(X_i,d_i)_i$. 

For morphisms $f,g:(X,d_X)\to (Y,d_Y)$ in $\textbf{Met}$, let $\sim$ be the smallest equivalence relation such that $y_1\sim y_2$ if there exit an $x\in X$ such that $f(x)=y_1$ and $g(x)=y_2$. Denote $F:=Y/\sim$. Define a map $d:F\times F\to [0,\infty]$ by \[d(y_1,y_2):=\inf\{d_Y(\tilde{y}_1,y)+d_Y(y,\tilde{y}_2)\mid y_1\sim \tilde{y}_1, y_2\sim \tilde{y}_2\}.\]
This map is well-defined and is a metric. The quotient map $Y\to F$ is $1$-Lipschitz and it is easy to verify that $F$ is the coequalizer of $f$ and $g$ in $\textbf{Met}$. 

By Proposition \ref{RN3.4}, $\mathbf{CMet}$ is reflective in $\textbf{Met}_1$ and therefore it is cocomplete. Colimits in $\mathbf{CMet}$ are constructed by reflecting colimits in $\textbf{Met}$.
\end{proof}
For (complete) metric spaces $(X_1,d_1)$ and $(X_2,d_2)$ let $(X_1,d_1)\otimes (X_2,d_2)$ be the (complete) metric space formed by the set $X_1\times X_2$ with the metric \[d((x_1,x_2),(y_1,y_2)):=d_1(x_1,y_1)+d(x_2,y_2).\] 
For $1$-Lipschitz maps $f_1:(X_1,d_{X_1})\to (Y_1,d_{Y_1})$ and $f_2:(X_2,d_{X_2})\to (Y_2,d_{Y_2})$, there is a $1$-Lipschitz map \[f_1\otimes f_2:(X_1,d_{X_1})\otimes (X_2,d_{X_2})\to (Y_1,d_{Y_1})\otimes d(Y_2,d_{Y_2}) \] defined by \[(x_1,x_2)\mapsto (f_1(x_1),f_2(x_2)).\]

This gives a functor $\otimes:\mathbf{CMet}\times \mathbf{CMet}\to \mathbf{CMet}$ and forms a symmetric monoidal product on $\mathbf{CMet}$.

For metric spaces $(X,d_X)$ and $(Y,d_Y)$, let $[(X,d_X),(Y,d_Y)]$ be the set of $1$-Lipschitz maps $(X,d_X)\to (Y,d_Y)$ together with the metric defined by \[d(f,g):=\sup\{d_Y(f(x),g(x))\mid x\in X\}.\]

\begin{prop}\label{RN3.6}
If $(Y,d_Y)$ is complete, then so is $[(X,d_X),(Y,d_Y)]$.
\end{prop}
\begin{proof}
    Let $(f_n)_n$ be a Cauchy sequence in $[(X,d_X),(Y,d_Y)]$, then it is clear that $(f_n(x))_n$ is a Cauchy sequence for every $x$. Therefore $(f_n(x))_n$ converges to an element $f_x\in Y$. 

    Define a map $f:X\to Y$ by sending $x$ to $f_x$. Consider $x_1$ and $x_2$ in $X$ and let $\epsilon>0$, there exists an $n\geq 1 $ such that  \begin{align*}d_Y(f(x_1),f(x_2))& \leq d_Y(f_{x_1},f_n(x_1))+d_Y(f_n(x_1),f_n(x_2))+d_Y(f_n(x_2),f_{x_2})\\
    & \leq d_Y(f_n(x_1),f_n(x_2)+\epsilon \leq d_X(x_1,x_2)+\epsilon.\end{align*}
    Taking $\epsilon\to 0$ shows that $f$ is $1$-Lipschitz. 

    For $\epsilon>0$, there is an $N\geq 1$ such that for $n_1,n_2\geq N$, \[d(f_{n_1},f_{n_2})\leq \epsilon.\]
    Let $n\geq N$. For $x\in X$, there exists $M_x\geq n$ such that $d_Y(f_{x},f_{M_x}(x))\leq \epsilon$ and therefore \[d_Y(f_{x},f_{n}(x))\leq d_Y(f_{x},f_{M_x}(x))+d_Y(f_{M_x}(x),f_n(x)) \leq 2\epsilon.\]
    Since $n$ is independent from $x$, we can conclude that $d(f,f_n)\leq 2\epsilon$ for all $n\geq N$, which means that $(f_n)_n$ converges to $f$ in $[(X,d_X),(Y,d_Y)]$.
\end{proof}
\begin{prop}\label{RN3.7}
    The monoidal category $\mathbf{CMet}$ is closed.
\end{prop}
\begin{proof}
    Let $(X,d_X),(Y,d_Y)$ and $(Z,d_Z)$ be complete metric spaces. It is easy to verify that there is a bijection \[\mathbf{CMet}(X\otimes Y,Z)\cong \mathbf{CMet}(X,[Y,Z]).\]
\end{proof}
\begin{prop}\label{RN3.8}
For metric spaces $X$ and $Y$, $\overline{X\otimes Y}\cong \overline{X}\otimes \overline{Y}$
\end{prop}
\begin{proof}
    Let $Z$ be a complete metric space. We have the following bijections: 
\begin{center}\begin{tabular}{c}
    $X\otimes Y  \to Z$ \\
    \hline
    $X  \to [Y,Z]$\\
    \hline 
    $\overline{X}\to [Y,Z]$\\
    \hline 
    $\overline{X}\otimes Y \to Z$\\ \hline
    $Y\to [\overline{X},Z]$\\ \hline
    $\overline{Y}\to [\overline{X},Z]$\\ \hline
    $\overline{X}\otimes \overline{Y}\to Z$
    \end{tabular}\end{center}

Here we used that by Proposition \ref{RN3.6}, $[Y,Z]$ and $[\overline{X},Z]$ are complete, since $Z$ is. Since $Z$ was chosen arbitrarily, the claim now follows.\end{proof}

Proposition \ref{RN3.7} says that $\mathbf{CMet}$ is a closed monoidal category. In what follows we will look at categories that are \emph{enriched} over this closed monoidal category. The 2-category of $\mathbf{CMet}$-enriched categories, enriched functors and enriched natural transformations is denoted as $\mathbf{CMet}\textbf{-Cat}$. 

The forgetful functor $U:\mathbf{CMet}\to \textbf{Set}$ induces a 2-functor $U_*:\mathbf{CMet}\textbf{-Cat}\to \textbf{Cat}$. Therefore, for $\mathbf{CMet}$-enriched categories $\mathcal{C}$ and $\mathcal{D}$, there is a functor \[\mathbf{CMet}\textbf{-Cat}[\mathcal{C},\mathcal{D}]\to \textbf{Cat}[U_*\mathcal{C},U_*\mathcal{D}].\]

The following lemmas will be used later to lift the results from section \ref{RNSection2} to the enriched setting.

\begin{lem}\label{RN3.9}
For a $\mathbf{CMet}$-enriched categories $\mathcal{C}$, the functor \[\mathbf{CMet}\mathbf{-Cat}[\mathcal{C},\mathbf{CMet}]\to \mathbf{Cat}[U_*\mathcal{C},U_*\mathbf{CMet}]\]
is full and faithful. 
\end{lem}

\begin{proof}
Let $F$ and $G$ be enriched functor $\mathcal{C}\to \mathbf{CMet}$. There is a one-to-one correspondence between $1$-Lipschitz maps $Fc\to Gc$ and $1\to [Fc,Gc]$ for all objects $c$ in $\mathcal{C}$. It follows now that every natural transformation $U_*F\to U_*G$ can be uniquely lifted to an enriched natural transformation $F\to G$.
\end{proof}
The following corollary states that if the non-enriched right Kan extension of enriched functors is an enriched functor, then it is also the enriched right Kan extension.

\begin{cor}\label{RN3.10}
    Let $\mathcal{C},\mathcal{D}$ be $\mathbf{CMet}$-enriched categories. Let $F:\mathcal{C}\to \mathbf{CMet}$, $G:\mathcal{C}\to \mathcal{D}$ and $H:\mathcal{D}\to \mathbf{CMet}$ be enriched functors and let $\epsilon:U_*H\circ U_*G\to U_*F$ be a (non-enriched) natural transformation such that \[\begin{tikzcd}
	{U_*{\mathcal{C}}} && {U_*\mathbf{CMet}} \\
	\\
	{U_*\mathcal{D}}
	\arrow[""{name=0, anchor=center, inner sep=0}, "{U^*F}", from=1-1, to=1-3]
	\arrow["{U_*G}"', from=1-1, to=3-1]
	\arrow["{U_*H}"', from=3-1, to=1-3]
	\arrow["\epsilon", shorten <=13pt, shorten >=9pt, Rightarrow, from=3-1, to=0]
\end{tikzcd}\]
$\epsilon$ exhibits $U_*H$ as the right Kan extension of $U_*F$ along $U_*G$, then there exists a unique enriched natural transformation $\tilde{\epsilon}:HG\to F$ such that $U_*\tilde{\epsilon}=\epsilon$ and $\Tilde{\epsilon}$ exhibits $H$ as the right Kan extension of $F$ along $G$.
\end{cor}

\subsection{\texorpdfstring{$\mathbf{Prob}$ is enriched over $\mathbf{CMet}$}{Prob is enriched over CMet}}\label{RNSection3.2}
Let $\bfomega_1:=(\Omega_1,\mathcal{F}_1,\mathbb{P}_1)$ and $\bfomega_2:=(\Omega_2,\mathcal{F}_2,\mathbb{P}_2)$ be probability spaces. Let $\mathrm{prob}[\bfomega_1,\bfomega_2]$ be the set of measure-preserving maps $\bfomega_1\to \bfomega_2$. For $f_1,f_2\in \mathrm{prob}[\bfomega_1,\bfomega_2]$, define \[d_{\bfomega_1,\bfomega_2}(f_1,f_2):=\sup\{\mathbb{P}(f_1^{-1}(A)\triangle f_2^{-1}(A))\mid A \text{ measurable subset of }\bfomega_2\}.\]
Note that that this is well-defined as it does not depend on the choice of representative. This turns $\mathrm{prob}[\bfomega_1,\bfomega_2]$ into a metric space. Proposition \ref{RN3.11} says that in the case that $\bfomega_2$ is a finite probability space, this metric space is complete.

\begin{prop} \label{RN3.11}
Let $\bfomega:=(\Omega,\mathcal{F},\mathbb{P})$ and $(A,p)$ be probability spaces, where $A$ is finite. Then $\mathrm{prob}[\bfomega,(A,p)]$ is a complete metric space. 
\end{prop}

\begin{proof}
Let $\mathrm{prob}[\bfomega,(A,p)]/=_\mathbb{P}$ be the metric space of \emph{equivalence classes} of $\mathbb{P}$-almost surely equal measure-preserving maps together with the metric induced by $d_{\bfomega,(A,p)}$\footnote{This metric space is the \emph{metric reflection} of the \emph{pseudo}metric space $\mathrm{prob}[\bfomega,(A,p)].$}.

To prove the claim, it is enough to show that $\mathrm{prob}[\bfomega,(A,p)]/=_\mathbb{P}$ is complete. We will show that $\mathrm{prob}[\bfomega,(A,p)]/=_\mathbb{P}$ is isomorphic, as a metric space, to a complete subspace of $\RV(\bfomega)^{\mathcal{P}(A)}$, the $\mathcal{P}(A)$-fold product of the metric space $\RV(\bfomega)$. There is a map $$\varphi:\mathrm{prob}[\bfomega,(A,p)]/=_\mathbb{P}\to \RV(\bfomega)^{\mathcal{P}(A)}$$ defined by $$f\mapsto \left(1_{f^{-1}(A')}\right)_{A'\subseteq A}.$$ This map is injective, well-defined and distance-preserving. To show that $\mathrm{Im}(\varphi)$ is closed, consider a sequence $(f_n)_n$ in $\mathrm{prob}[\bfomega_1,(A,p)]/_\mathbb{P}$ such that $\varphi(f_n)$ converges to $g:=(g_{A'})_{A'\subseteq A} \in \RV(\bfomega_1)^{\mathcal{P}(A)}.$ Then $g_{A'}$ is ($\mathbb{P}$-almost surely) equal to $1_{E_{A'}}$ for some measurable subset $E_{A'}$. We also have that $g_{A_1}g_{A_2}=g_{A_1\cap A_2}$ and therefore $$\mathbb{P}(E_{A_1}\cap E_{A_2})=\mathbb{P}(E_{A_1\cap A_2}).$$
In particular we have that $\mathbb{P}\left (E_{\{a_1\}}\cap E_{\{a_2\}}\right)=0$  for $a_1\not=a_2$, since $g_\emptyset=\lim_{n\to \infty}\varphi(f_n)_\emptyset=0.$ Therefore, we can assume without loss of generality that $\left(E_{\{a\}}\right)_{a\in A}$ are pairwise disjoint. 

Furthermore, for every $a\in A$, $$\mathbb{P}\left(E_{\{a\}}\right)=\lim_{n\to \infty}\int 1_{f_n^{-1}(\{a\})}\mathrm{d}\mathbb{P}=p_a$$

It follows now that there exists a measure-preserving map $f:\bfomega_1\to (A,p)$ such that $f^{-1}(\{a\})=E_{\{a\}}$, which is $\mathbb{P}$-almost surely unique. This means that $f\in \mathrm{prob}[\bfomega,(A,p)]$ and $\varphi(f)=g.$

We have shown that $\mathrm{prob}[\bfomega,(A,p)]$ is isomorphic, as a metric space, to a closed subset of $\RV(\bfomega)^{\mathcal{P}(A)}$. Since $\mathcal{P}(A)$ is finite, $\RV(\bfomega)^{\mathcal{P}(A)}$ is complete. We can conclude now that $\mathrm{prob}[\bfomega,(A,p)]$ is complete as well. 
\end{proof}
For probability spaces $\bfomega_1$ and $\bfomega_2$ let $\mathrm{Prob}_r[\bfomega_1,\bfomega_2]$ be the metric space we obtain by scaling the metric of $\mathrm{prob}[\bfomega_1,\bfomega_2]$ by a factor $r>0$. Moreover, let $\mathbf{Prob}_r[\bfomega_1,\bfomega_2]$ be the completion of $\mathrm{Prob}_r[\bfomega_1,\bfomega_1]$. 
Note that the limit of the diagram 
\[\begin{tikzcd}
	{\mathbf{Prob}_1[\bfomega_1,\bfomega_2}] & \ldots &{\mathbf{Prob}_r[\bfomega_1,\bfomega_2}] & \ldots
	\arrow[from=1-2, to=1-1]
	\arrow[from=1-3, to=1-2]
 	\arrow[from=1-4, to=1-3]
\end{tikzcd}\]
is the discrete pseudometric space of measure preserving maps $f:\bfomega_1\to \bfomega_2$. 

Consider probability spaces $\bfomega_1, \bfomega_2$ and $\bfomega_3$. Sending a pair of measure-preserving maps $f:\bfomega_1\to \bfomega_2$ and $g:\bfomega_2\to \bfomega_3$, to $g\circ f:\bfomega_1\to \bfomega_3$, defines a $1$-Lipschitz map \[-\circ -: \mathrm{prob}[\bfomega_1,\bfomega_2]\otimes \mathrm{prob}[\bfomega_2,\bfomega_3]\to \mathrm{prob}[\bfomega_1,\bfomega_3]\]
Indeed, this follows from the fact that \begin{align*}\mathbb{P}_1((g_1f_1)^{-1}(A)\triangle (g_2f_2)^{-1}(A))& \leq \mathbb{P}_1((g_1f_1)^{-1}(A)\triangle (g_2f_1)^{-1}(A))+\mathbb{P}_1((g_2f_1)^{-1}(A)\triangle (g_2f_2)^{-1}(A))\\
& = \mathbb{P}_2(g_1^{-1}(A)\triangle g_2^{-1}(A))+\mathbb{P}_1(f_1^{-1}(g_2^{-1}(A))\triangle f_2^{-1}(g_2^{-1}(A)))\\
& \leq d_{\bfomega_2,\bfomega_3}(g_1,f_2)+d_{\bfomega_1,\bfomega_2}(f_1,f_2)
\end{align*}
for all measurable subsets of $\bfomega_3$. Using Proposition \ref{RN3.8}, this induces a $1$-Lipschitz map 
\[\mathbf{Prob}_r[\bfomega_1,\bfomega_2]\otimes \mathbf{Prob}_r[\bfomega_2,\bfomega_3]\to \mathbf{Prob}_r[\bfomega_1,\bfomega_3]\]

The above describes a category enriched over $\mathbf{CMet}$, whose objects are probability spaces and whose hom-objects are given by $\mathbf{Prob}_r[\bfomega_1,\bfomega_2]$ for probability spaces $\bfomega_1$ and $\bfomega_2$. We denote this category by $\mathbf{Prob}_r$. The subcategory of finite probability spaces is denoted as $\mathbf{Prob}_r^f$ and clearly there is an enriched inclusion functor $i_r:\mathbf{Prob}_r^f\to \mathbf{Prob}_r$. Furthermore, note that $U^*\mathbf{Prob}_r$ is the (non-enriched) category $\mathbf{Prob}$ for all $r>0$.
\subsection{\texorpdfstring{The enriched functors $\M_r$ and $\RV_r$}{The enriched functors Mr and RVr}}\label{RNSection3.3}
In this section we will show that everything proved in section \ref{RNSection2} still works in the enriched context. 

The (non-enriched) functor $\M_r:\textbf{Prob}\to \mathbf{CMet}$ from Section \ref{2} induces a $\mathbf{CMet}$-enriched functor $\textbf{Prob}_r\to\mathbf{CMet}$. Indeed, the assignment $f\mapsto \M_r(f)$, induces a \emph{$1$-Lipschitz map} \[\mathrm{prob}_r[\bfomega_1,\bfomega_2]\to [\M_r(\bfomega_1),\M_r(\bfomega_2)].\]
To see this, consider two measure preserving maps $f_1,f_2:\bfomega_1\to \bfomega_2$. For $\mu\in \M_r(\bfomega_1)$ and a measurable subset $A$ of $\bfomega_2$, we find that \begin{align*}
    \lvert \M_r(f_1)(\mu)(A)-\M_r(f_2)(\mu(A))\rvert & = \lvert \mu(f_1^{-1}(A))-\mu(f_2^{-1}(A)))\rvert\\
    & = \left\lvert\int 1_{f_1^{-1}(A)}-1_{f_2^{-1}(A)} \text{d}\mu\right\rvert\\
    &\leq \mu\left(f_1^{-1}(A)\triangle f_2^{-1}(A)\right) \\
    &\leq r\mathbb{P}_1(f_1^{-1}(A)\triangle f_2^{-1}(A))
    \leq rd_{\Omega_1,\Omega_2}(f_1,f_2)
\end{align*}
Taking the supremum over all measurable subsets $A$ of $\bfomega_2$, gives us that \[d_{\M_r(\bfomega_2)}(\M_r(f_1)(\mu),\M_r(f_2)(\mu)) \leq rd_{\bfomega_1,\bfomega_2}(f_1,f_2).\]
Finally, by taking the supremum of over all $\mu\in \M_r(\bfomega_1)$, we see that the assignment $f\mapsto \M_r(f)$ defines a $1$-Lipschitz map. This gives a $1$-Lipschitz map $\mathbf{Prob}_r[\bfomega_1,\bfomega_2]\to [\M_r(\bfomega_1),\M_r(\bfomega_2)]$. The obtained enriched functor $\textbf{Prob}_r\to \mathbf{CMet}$ is also denoted by $\M_r$.

The restriction to finite probability spaces is also by $\M_r^f$ and is an enriched functor since it is the composition of the enriched functors $\M_r:\textbf{Prob}_r\to \mathbf{CMet}$ and $i_r:\textbf{Prob}_r^f\to \textbf{Prob}_r$.

\begin{prop}\label{RN3.12}
The commutative triangle of enriched functors 
\[\begin{tikzcd}
	{\mathbf{Prob}_r^f} && {\mathbf{CMet}} \\
	{\mathbf{Prob}_r} 
	\arrow["{\M_r^f}", from=1-1, to=1-3]
	\arrow["{i_r}"', from=1-1, to=2-1]
	\arrow["{\M_r}"', from=2-1, to=1-3]
\end{tikzcd}\]
exhibits $\M_r$ as the right Kan extension of $\M_r^f$ along $i_r$. 
\end{prop}
\begin{proof}
This follows from Theorem \ref{RN2.6} together with Corollary \ref{RN3.10}. 
\end{proof}

The (non-enriched) functor $\RV_r:\textbf{Prob}\to \mathbf{CMet}$ from section \ref{RNSection2} induces a $\mathbf{CMet}$-enriched functor $\textbf{Prob}_r\to\mathbf{CMet}$. Indeed, the assignment $f\mapsto \RV_r(f)$ induces a \emph{$1$-Lipschitz map}
\[\mathrm{prob}(\bfomega_1,\bfomega_2)\to [\RV_r(\bfomega_1),\RV_r(\bfomega_2)].\]

To see this, consider two measure preserving maps $f_1,f_2:\bfomega_1\to \bfomega_2$. For $X\in \RV_r(\bfomega_1)$, consider the measurable subset $A^+:=\{\mathbb{E}[X\mid f_1]\geq \mathbb{E}[X\mid f_2]\}$ and let $A^-$ be its complement.\footnote{The subset $A^+$ should actually be defined as $\{g_1\geq g_2\}$ for some measurable maps $g_1:\Omega_2\to [0,1]$ and $g_2:\Omega_2\to [0,n]$ representing $\mathbb{E}[X\mid f_1]$ and $\mathbb{E}[X\mid f_2]$ respectively. However, everything that follows is independent from the choice of $g_1$ and $g_2$ and therefore we just write $\{\mathbb{E}[X\mid f_1]\geq \mathbb{E}[X\mid f_2]\}$.} We now find that \begin{align*}
    2 d_{\RV_r(\bfomega_2)}(\RV_r(f_1)(X),\RV_r(f_2)(X))&= \mathbb{E}[\lvert \mathbb{E}[X\mid f_1]-\mathbb{E}[X\mid f_2]\rvert]\\
    & = \mathbb{E}[(\mathbb{E}[X\mid f_1]-\mathbb{E}[X\mid f_2])1_{A^+}]+\mathbb{E}[(\mathbb{E}[X\mid f_2]-\mathbb{E}[X\mid f_1])1_{A^-}]\\
    &= \mathbb{E}[X(1_{f_1^{-1}(A^+)}-1_{f_2^{-1}(A^+)}+1_{f_2^{-1}(A^-)}-1_{f_1^{-1}(A^-)})]\\
    &\leq r (\mathbb{E}[\lvert 1_{f_1^{-1}(A^+)}-1_{f_2^{-1}(A^+)}\rvert ]+\mathbb{E}[\lvert 1_{f_1^{-1}(A^-)}-1_{f_2^{-1}(A^-)}\rvert ])\\
    &= r (\mathbb{P}(f_1^{-1}(A^+)\triangle f_2^{-1}(A^+))+\mathbb{P}(f_1^{-1}(A^-)\triangle f_2^{-1}(A^-)))\\
    &\leq 2r d_{\bfomega_1,\bfomega_2}(f_1,f_2)
    \end{align*}

By taking the supremum over all $X\in \RV_r(\bfomega_1)$, we see that the assignment $f\mapsto \RV_r(f)$ defines a $1$-Lipschitz maps. This induces a $1$-Lipschitz map $\mathbf{Prob}_r[\bfomega_1,\bfomega_2]\to [\RV_r(\bfomega_1,\RV_r(\bfomega_2)]$. The enriched functor $\mathbf{Prob}_r\to \mathbf{CMet}$ that we obtain will also be denoted by $\RV_r$.

The restriction to finite probability spaces is also denoted by $\RV_r^f$ and is an enriched functor since it is the composition of the enriched functors $\RV_r:\mathbf{Prob}_r\to \mathbf{CMet}$ and $i_r:\mathbf{Prob}^f_r\to \mathbf{Prob}_r$.

\begin{prop}\label{RN3.13}
The commutative triangle of \emph{enriched} functors 
\[\begin{tikzcd}
	{\mathbf{Prob}_r^f} && {\mathbf{CMet}} \\
	{\mathbf{Prob}_r} 
	\arrow["{\RV_r^f}", from=1-1, to=1-3]
	\arrow["{i_r}"', from=1-1, to=2-1]
	\arrow["{\RV_r}"', from=2-1, to=1-3]
\end{tikzcd}\]
exhibits $\RV_r$ as the right Kan extension of $\RV_r^f$ along $i_r$. 
\end{prop}
\begin{proof}
This follows from Theorem \ref{RN2.10} together with Corollary \ref{RN3.10}. 
\end{proof}

\subsection{The martingale convergence theorem}\label{RNSection3.4}
Let $H:\textbf{Prob}_r\to \mathbf{CMet}$ be an enriched functor such that the commutative triangle \[\begin{tikzcd}
	{\textbf{Prob}_r^f} && {\mathbf{CMet}} \\
	{\textbf{Prob}_r}
	\arrow["{H\circ i_r}", from=1-1, to=1-3]
	\arrow["{i_r}"', from=1-1, to=2-1]
	\arrow["H"', from=2-1, to=1-3]
\end{tikzcd}\]
exhibits $H$ as the right Kan extension of $H\circ i_r$ along $i_r$. 

Let $(\Omega, \left(\mathcal{F}_i\right)_{i\in I},\mathcal{F},\mathbb{P})$ be a filtered probability space. For $i\in I$, we write $\bfomega_i$ for $(\Omega,\mathcal{F}_i,\mathbb{P}\mid_{\mathcal{F}_i})$  and $\bfomega$ for $(\Omega, \mathcal{F},\mathbb{P})$. For $i\leq j$ in $I$ there is a measure-preserving map $f_{ij}:\bfomega_j\to \bfomega_i$ and for $i\in I$ there is a measure-preserving map $f_i:\bfomega\to \bfomega_i$. This induces a diagram $D_\bfomega$ in the underlying category of $\mathbf{Prob}_r$ of which $\bfomega$ is the conical limit. 


In the case where $I=\mathbb{N}$, we have the following diagram.

\[\begin{tikzcd}
	{\bfomega_1} & {\bfomega_2} & {\bfomega_3} & \dots & {\bfomega}
	\arrow[from=1-2, to=1-1]
	\arrow[from=1-3, to=1-2]
	\arrow[from=1-4, to=1-3]
	\arrow[from=1-5, to=1-4]
\end{tikzcd}\]

\begin{lem}\label{RN3.14}
Let $(\Omega, \left(\mathcal{F}_i\right)_{i\in I},\mathcal{F},\mathbb{P})$ be a filtered probability space. Then for $E\in \mathcal{F}$, there exists a sequence $(E_n)_n$ in $\bigcup_{i\in I} \mathcal{F}_i$ such that \[\mathbb{P}(E\triangle E_n)\to 0.\]
\end{lem}

\begin{proof}
Note that $\bigcup_{i\in I}\mathcal{F}_i$ is closed under complements and \emph{finite} intersections and unions, because $I$ is directed. The claim now immediately follows from Theorem D in Section 13 in \cite{halmos}.
\end{proof}

\begin{cor}\label{RN3.15}
Let $\mathbf{A}$ be a finite probability space. The enriched functor $\mathbf{Prob}_r(-,\mathbf{A}):\mathbf{Prob}_r\to \mathbf{CMet}^\text{op}$ preserves the limit of $D_\bfomega$.
\end{cor}
\begin{proof}

 We will only prove this for the case that $\mathbf{A}$ is a probability space with two elements. We will write $\mathbf{2}$ instead of $\mathbf{A}$. It is enough to show that \[\bigcup_{i\in I}\textbf{Prob}_r(\mathbf{\Omega}_i,\textbf{2})\subseteq \textbf{Prob}_r(\mathbf{\Omega},\textbf{2})\] 
    is a dense subset. 

By Proposition \ref{RN3.11}, we know that an element of $\mathbf{Prob}_r(\bfomega,\textbf{2})$ is a measure-preserving map $f:\bfomega\to \textbf{2}$. Let $f:\mathbf{\Omega}\to \mathbf{2}$ be an element of $\mathbf{Prob}_r(\mathbf{\Omega},\mathbf{2})$. This can be identified by measurable subset $E$ of $\Omega$. By Lemma \ref{RN3.14}, there is a sequence $(E_m)_m$ in $\bigcup_{i\in I}\mathcal{F}_i$ such that $\mathbb{P}(E\triangle E_n)\to 0$. Every $E_m$ can be identified by a measure preserving map $f_m:\mathbf{\Omega}_{i_n}\to \textbf{2}$. Because $d_{\Omega,2}(f,f_m)=r\mathbb{P}(E\triangle E_m)$, the result now follows. 
\end{proof}
\begin{thm}\label{RN3.16}
    The functor $H:\mathbf{Prob}_r\to \mathbf{CMet}$ preserves the limit of $D_\bfomega: I\to \mathbf{Prob}$.
\end{thm}
\begin{proof}
    Since $H$ is the right Kan extension of $H\circ i_r$ along $i_r$, it can be represented as a weighted limit, as explained in section 4.1 in \cite{kelly}. \[H(\mathbf{\Omega})\cong\left\{\textbf{Prob}_r(\mathbf{\Omega},i_r-),H\circ i_r\right\}.\] Using Corollary \ref{RN3.15} and using properties of weighted limits we find \begin{align*}
    H(\mathbf{\Omega})& =\left\{\textbf{Prob}_r(\mathbf{\Omega},i_r-),H\circ i_r\right\}\\
    &= \left\{\text{colim}_i\textbf{Prob}_r(\mathbf{\Omega}_i,i_r-),H\circ i_r\right\}\\
     &= \lim_i\left\{\textbf{Prob}_r(\mathbf{\Omega}_i,i_r-),H\circ i_r\right\}\\
     &=\lim_iH(\mathbf{\Omega}_i)
\end{align*} \end{proof} By Proposition \ref{RN3.12} and Proposition \ref{RN3.13}, we know that both $\M_r$ and $\RV_r$ are right Kan extensions of their restrictions to $\textbf{Prob}_r^f$ along $i_r:\mathbf{Prob}_r^f\to \mathbf{Prob}_r$. Therefore we can now apply Theorem \ref{RN3.16} to the functors $\M_r$ and $\RV_r$. 

The functor $\RV_r$ preserves the limit of $D_\bfomega$. The limit of $\RV_rD_\bfomega$ can be constructed in the usual way we construct cofiltered limits in $\mathbf{CMet}$.

The underlying set of the limit of $\mathrm{RV}_r D_\bfomega$ is given by $$\left\{ (X_i)_{i\in I}\in \prod_{i\in I}\RV_r(\bfomega_i) \mid \RV_r(f_{ij})(X_j)=X_i \text{ for all }i\leq j\right\}$$
which is equal to $$\left\{(X_i)_{i\in I}\in \prod_{i\in I}\RV_r(\bfomega_i)\mid \mathbb{E}[X_j\mid f_{ij}]=X_i \right\}.$$ This means that the underlying set of $\lim\RV_rD_\bfomega$ is precisely the collection of martingales, uniformly bounded by $r$, on the filtered probability space $(\Omega,\left(\mathcal{F}_i\right)_{i\in I},\mathcal{F},\mathbb{P})$. Theorem \ref{RN3.16} now says that the map $$\RV_r(\bfomega)\to \lim_i \RV_r(\bfomega_i)$$ defined by the assignment $$X\mapsto \left(\RV_r(f_i)(X)\right)_{i\in I}= \left( \mathbb{E}[X\mid f_i]\right)_{i\in I}$$ is an isomorphism. In other words, for every martingale $(X_i)_i$ there is a $\mathbb{P}$-almost surely unique random variable $X\in \RV_r(\bfomega)$ such that $$\mathbb{E}[X\mid f_i]=X_i.$$ This proves, \emph{categorically}, the following weaker martingale convergence theorem. 

\begin{thm}\label{RN3.17}
    Let $(X_i)_{i\in I}$ be a martingale such that for all $i\in I$, $$\mathbb{P}(X_i\leq r)=1.$$
    Then there exists a unique $X\in \RV_r(\bfomega)$ such that for all $i\in I$, $$\mathbb{E}[X\mid f_i]=X_i.$$
\end{thm}

Theorem \ref{RN3.16} also implies that the functor $\M_r$ preserves the limit of $D_\bfomega$. Also for this functor we can construct the cofiltered limit $\lim\M_rD_\bfomega$ in the usual way; its underlying set is given by $$\left\{(\mu_i)_{i\in I}\in \prod_{i\in I}\M_r(\bfomega_i)\mid \M_r(f_{ij})(\mu_j)=\mu_i \text{ for all }i\leq j\right\}$$
which is equal to $$\left\{(\mu_i)_{i\in I}\in \prod_{i\in I}\M_r(\bfomega_i)\mid \mu_j\mid_{\mathcal{F}_i}=\mu_i \text{ for all } i\leq j \right\}$$

Theorem \ref{RN3.16} says that the map $$\M_r(\bfomega)\to \lim_i\M_r(\bfomega_i)$$ defined by the assignment $$\mu \mapsto (\M_r(f_i)(\mu))_{i\in I} = \left(\mu\mid_{\mathcal{F}_i}\right)_{i\in I}$$ is an isomorphism. Therefore, for every family $(\mu_i)_{i\in I}$ of measures, where $\mu_i\in \M_r(\bfomega_i)$ such that for all $i\leq j$, $$\mu_i\mid_{\mathcal{F}_j}=\mu_j,$$
there exists a unique $\mu \in \M_r(\bfomega)$ such that $$\mu\mid_{\mathcal{F}_i}=\mu_i.$$ This gives a categorical proof for the following version of the Kolmogorov extension theorem.

\begin{thm}\label{RN3.18}
Consider a family $(\mu_i)_{i\in I}$ such that $\mu_i$ is a measure on $\bfomega_i$ and $\mu_i\leq r\mathbb{P}$. Suppose that for all $i\leq j$, $$\mu_j\mid_{\mathcal{F}_i}=\mu_i.$$
Then there exist a unique measure $\mu$ on $\bfomega$ with $\mu \leq r\mathbb{P}$ such that for all $i\in I$, $$\mu\mid_{\mathcal{F}_i}=\mu_i.$$
\end{thm}
\begin{rem}\label{RN3.19}
Theorem \ref{RN3.16} implies an even stronger result than Theorem \ref{RN3.17} and \ref{RN3.18}. It not only says that for every martingale $(X_i)_i$ there exists a random variable $X$ such that $\mathbb{E}[X\mid f_i]=X_i$, but it says that this happens in an \emph{isometric} way. In other words, $\sup_id(X_i,Y_i)=d(X,Y)$ for martingales $(X_i)_i$ and $(Y_i)_i$ and their corresponding limiting random variables $X$ and $Y$. In a similar way, we have that the Kolmogorov extension from Theorem \ref{RN3.18} is isometric too. 

Furthermore, for a \emph{consistent} family of measures $(\mu_i)_i$ and its limiting measure $\mu$ such as in Theorem \ref{RN3.17}, the collection of Radon-Nikodym derivates $\left(\frac{\text{d}\mu_i}{\text{d}\mathbb{P}}\right)_i$ form a martingale and the limiting random variable of this martingale is the Radon-Nikodym derivative of $\mu$ with respect to $\mathbb{P}$. 
\end{rem}
\begin{rem}\label{RN3.20}
An alternative approach we could have taken is to define the category $\mathbf{Prob}$ as the category of probability spaces and \emph{equivalence classes} of almost surely equal measure-preserving maps. Using this approach, we would not have to deal with the \emph{pseudo}metric spaces. However, because  not every measurable space has a measurable diagonal, sets of the form $\{f=g\}$ are not necessarily measurable. We would therefore need to use outer measures to define the concept of almost surely equal maps. 
\end{rem}

\begin{rem}[Future work]\label{RN3.21}
Since Theorem \ref{RN3.16} holds for an arbitrary enriched functor $H$, such that it is the right extension of its restriction to $\mathbf{Prob}_r^f$ along $i_r$, we expect that this result can be used to obtain categorical proofs of other martingale convergence theorem and extensions results.

Moreover, we are interested in investigating to what extent we can replace $[0,1]$ and $[0,\infty)$ by arbitrary algebras of probability monads and measure monads on categories of complete metric spaces. This would lead to a martingale theory of generalized random variables, taking values in an algebra of a certain probability monad or measure monad.
\end{rem}
\newpage
\begin{appendices}
\section{Inequalities in the proof of Theorem \ref{RN2.10}}
\begin{lem}\label{appendix}
Let $(\Omega,\mathcal{F},\mathbb{P})$ be a probability space and let $(A,p)$ and $(B,q)$ be finite probability spaces. Let $f:(\Omega,\mathcal{F},\mathbb{P})\to (A,p)$ and $g:(\Omega,\mathcal{F},\mathbb{P})\to (B,q)$ be measure-preserving maps. Let $s:(A,p)\to (B,q)$ be a measure-preserving map such that $sf=g$. Let $s_g^y$ and $s_f^y$ be such as in Theorem \ref{RN2.10}. The following equalities hold.
\begin{enumerate}[label=(\roman*)]
    \item \label{i}$s_g^ys_f^y=\sum_{b\in B}q_g(y)(b)\sum_{a\in s^{-1}(b)}q_f(y)(a)1_{f^{-1}(a)}.$
    \item \label{ii}$\mathbb{E}[s_f^ys_g^y]=\sum_{b\in B}q_g(y)(b)^2q_b$.
    \item \label{iii}$(s_g^y)^2=\sum_{b\in B}q_g(y)(b)^21_{g^{-1}(b)}$ and $(s_f^y)^2=\sum_{a\in A}q_f(y)(a)^21_{g^{-1}(a)}$.
    \item \label{iv}$\mathbb{E}[(s_g^y)^2]=\sum_{b\in B}q_g(y)(b)^2q_b$ and $\mathbb{E}[(s_f^y)^2]=\sum_{a\in A}q_f(y)(a)^2p_a$.
    \item \label{v}$\mathbb{E}[(s_g^y)^2] \leq \mathbb{E}[(s_f^y)^2]$.
    \item \label{vi} $\mathbb{E}[(s_f^y)^2]-\mathbb{E}[(s_g^y)^2]=\mathbb{E}[(s_f^y-s_g^y)^2]$.
\end{enumerate}
\end{lem}
\begin{proof}
For \ref{i}, 
\begin{align*}
    s_g^ys_f^y &=\sum_{a\in A,b\in B}q_f(y)(a)q_g(y)(b)1_{f^{-1}(a)\cap g^{-1}(b)}\\
    &= \sum_{a\in A,b\in B}q_f(y)(a)q_g(y)(b)1_{f^{-1}(s^{-1}(b)\cap \{a\})} \\
    &=\sum_{b\in B}\sum_{a\in s^{-1}(b)}q_f(y)(a)q_g(y)(b)1_{f^{-1}(a)} \\
    & =\sum_{b\in B} q_g(y)(b) \sum_{a\in s^{-1}(b)} q_f(y)(a)1_{f^{-1}(a)}
\end{align*}
Integration \ref{i} gives us \[\mathbb{E}[s_f^ys_g^y]=\sum_{b\in B}q_g(y)(b)\sum_{a\in s^{-1}(b)}q_f(y)(a)p_a = \sum_{b\in B}q_g(y)(b)q_g(y)(b)q_b\]
This implies \ref{ii}. The result \ref{iii} and \ref{iv} follow from \ref{i} an \ref{ii}.
For $q_b\not=0$, 
\[q_g(y)(b)^2=\left(\sum_{a\in s^{-1}(b)}q_f(y)\frac{p_a}{q_b}\right)^2\leq \sum_{a\in s^{-1}(b)}q_f(y)^2\frac{p_a}{q_b}.\]
Here we used that $x\mapsto x^2$ defines a convex function. Multiplying both sides by $q_b$ and summing over $b\in B$, gives us \ref{v}.
For \ref{vi}
\begin{align}
    \mathbb{E}[(s_f^y)^2]& =\mathbb{E}[(s_f^y-s_g^y+s_g^y)^2]\\
   &=\mathbb{E}[(s_f^y-s_g^y)^2] + \mathbb{E}[(s_g^f)^2] + 2 (\mathbb{E}[s_f^ys_g^2]-\mathbb{E}[(s_g^y)^2])
    & = \mathbb{E}[(s_f^y-s_g^y)^2] + \mathbb{E}[(s_g^f)^2].
\end{align}
In the fourth equality we used that $\mathbb{E}[s_f^ys_g^2]-\mathbb{E}[(s_g^y)^2]=0$ by \ref{ii} and \ref{iv}. 
\end{proof}
The first inequality in the proof Theorem \ref{RN2.10} that we needed to show is exactly Lemma \ref{appendix}\ref{v}. For the second inequality, we combine Lemma \ref{appendix}\ref{vi} with Jensen's inequality. 
\end{appendices}
\bibliographystyle{abbrv}
\bibliography{ref.bib}
\end{document}